\documentclass{amsart}

\usepackage{allneeded}

\begin{document}

\title[Symplectically degenerate maxima]{Symplectically degenerate maxima\\ via generating functions}%

\author{Marco Mazzucchelli}%

\address{UMPA, \'Ecole Normale Sup\'erieure de Lyon, 69364 Lyon Cedex 07, France}%
\email{marco.mazzucchelli@ens-lyon.fr}%


\date{July 13, 2012. \emph{Revised:} November 13, 2012}%
\dedicatory{Dedicated to John Mather on the occasion of his 70th birthday.}

\begin{abstract}
We provide a simple proof of a theorem due to Nancy Hingston, asserting that symplectically degenerate maxima of any Hamiltonian diffeomorphism $\phi$ of the standard symplectic $2d$-torus are non-isolated contractible periodic points or their action is a non-isolated point of the average-action spectrum of $\phi$. Our argument is based on generating functions.
\end{abstract}

\maketitle

\begin{quote}
\begin{footnotesize}
\tableofcontents
\end{footnotesize}
\end{quote}

\section{Introduction}

A remarkable feature of Hamiltonian systems on symplectic manifolds is that they tend to have many periodic orbits. For autonomous systems, this was already noticed by Poincar\'e in one of his so-called ``other'' conjectures:

\vspace{5pt}

\begin{quote}
\selectlanguage{francais}
\emph{``Il y a m\^eme plus : voici un fait que je n'ai pu d\'emontrer rigoureusement, mais qui me para\^it pourtant tr\`es vraisemblable. \'Etant donn\'ees des \'equations de la forme d\'efinie dans le num\'ero 13 \textnormal{[i.e.\ autonomous Hamilton equations]} et une solution particuli\`ere quelconque de ces \'equations, on peut toujours trouver une solution p\'eriodique (dont la p\'eriode peut, il est vrai, \^etre tr\`es longue), telle que la diff\'erence entre les deux solutions soit aussi petite qu'on le veut, pendant un temps aussi long qu'on le veut. D'ailleurs, ce qui nous rend ces solutions p\'eriodiques si pr\'ecieuses, c'est qu'elles sont, pour ainsi dire, la seule br\`eche par o\`u nous puissions essayer de p\'en\'etrer dans une place jusqu'ici r\'eput\'ee inabordable.''}

\flushright Henri Poincar\'e, \cite[page 82]{Poincare:Les_methodes_nouvelles_de_la_mecanique_celeste_Tome_I}.%
\selectlanguage{english}%
\end{quote}

\vspace{5pt}

For non-autonomous systems the abundance of periodic motion is a global phenomenon somehow reflecting the fact that periodic orbits can be detected by a variational principle: they are critical points of the action functional. In 1984, Conley \cite{Conley:Lecture_at_the_University_of_Wisconsin} conjectured that every Hamiltonian diffeomorphism of the standard symplectic torus $(\T^{2d},\omega)$ has infinitely many contractible periodic points (a  periodic point is contractible when it corresponds to a contractible periodic orbit of any associated Hamiltonian system, see Section~\ref{s:discrete_symplectic_action}). The weaker version of this conjecture, for a \emph{generic} Hamiltonian diffeomorphism, was soon established by Conley and Zehnder \cite{Conley_Zehnder:Subharmonic_solutions_and_Morse_theory} and then extended to all the closed symplectically aspherical  manifolds by Salamon and Zehnder \cite{Salamon_Zehnder:Morse_theory_for_periodic_solutions_of_Hamiltonian_systems_and_the_Maslov_index}. The general Conley conjecture is significantly  harder. In dimension 2, for all the closed surfaces other than the sphere, it was established by Franks and Handel \cite{Franks_Handel:Periodic_points_of_Hamiltonian_surface_diffeomorphisms} and further extended to the category of Hamiltonian homeomorphisms by Le Calvez \cite{LeCalvez:Periodic_orbits_of_Hamiltonian_homeomorphisms_of_surfaces}. In higher dimension, the original conjecture for tori was recently established by Hingston \cite{Hingston:Subharmonic_solutions_of_Hamiltonian_equations_on_tori} and further generalized to closed symplectically aspherical manifolds by Ginzburg \cite{Ginzburg:The_Conley_conjecture}. Further extensions,  refinements and related results are contained in \cite{Long:Multiple_periodic_points_of_the_Poincare_map_of_Lagrangian_systems_on_tori, Lu:The_Conley_conjecture_for_Hamiltonian_systems_on_the_cotangent_bundle_and_its_analogue_for_Lagrangian_systems, Mazzucchelli:The_Lagrangian_Conley_conjecture, Ginzburg_Gurel:Local_Floer_homology_and_the_action_gap, Ginzburg_Gurel:Action_and_index_spectra_and_periodic_orbits_in_Hamiltonian_dynamics, Hein:The_Conley_conjecture_for_irrational_symplectic_manifolds, Hein:The_Conley_conjecture_for_the_cotangent_bundle, Ginzburg_Gurel:Conley_Conjecture_for_Negative_Monotone_Symplectic_Manifolds}.

The general strategy of the proof of the Conley conjecture goes along the following lines.  A contractible periodic  point with period $p$ corresponds to a contractible $p$-periodic orbit  of an associated Hamiltonian system, which in turn is a critical point of the Hamiltonian action functional in period $p$. For homological reasons, for every period $p$ one can detect critical points of the action functional with non-trivial local homology in Maslov degree $d$,  where $d$ is the half-dimension of the symplectic manifold. If one assumes that the Hamiltonian diffeomorphism has only finitely many  contractible periodic points, one among these points must be homologically visible in degree $d$ for infinitely many periods. This means that for infinitely many periods $p$, the $p$-periodic orbit corresponding to the periodic point has non-trivial local homology in Maslov degree $d$. By replacing the considered Hamiltonian flow with another one having the same time-$p$ map, one can further realize these orbits as  sufficiently flat local maxima of the Hamiltonian function at every time: for this reason, orbits of this kind are often called symplectically degenerate maxima\footnote{In \cite{Hingston:Subharmonic_solutions_of_Hamiltonian_equations_on_tori}, Hingston calls them topologically degenerate orbits instead.}. The second part or the proof of the Conley conjecture is the crucial one: the existence of a symplectically degenerate maximum $z$ automatically implies the existence of infinitely many other contractible periodic points. Indeed, if $z$ is isolated in the set of periodic points, its average-action turns out to be a non-isolated point in the average-action spectrum of the Hamiltonian diffeomorphism (see Section~\ref{s:discrete_symplectic_action} for the precise definition).

The first part of the proof, namely the detection of a symplectically degenerate maximum, is well understood: it can be carried out for the case of tori by means of Morse theory for the Hamiltonian action in a suitable functional setting, whereas for  more general symplectic manifolds one has to use Floer theory (see e.g.~\cite{Salamon:Lectures_on_Floer_homology}), a special Morse theory for the Hamiltonian action. The second part of the proof, namely the statement asserting that the presence of a  symplectically degenerate maximum implies the existence of infinitely many other periodic points, is much less transparent. For the torus, Hingston \cite{Hingston:Subharmonic_solutions_of_Hamiltonian_equations_on_tori} argued by a sophisticated analysis of the behavior of the action functional in a small neighborhood of a  symplectically degenerate maximum. For more general symplectic manifolds, Ginzburg \cite{Ginzburg:The_Conley_conjecture} used several subtle properties of Floer homology. As already pointed out by Hingston in \cite{Hingston:Subharmonic_solutions_of_Hamiltonian_equations_on_tori}, it is not known whether the new average-action values detected around the average action of a symplectically degenerate maximum $z$  actually correspond to periodic points accumulating at $z$.

The motivation for the current work is to give a simple and more transparent argument for the second part of the proof of the Conley conjecture in the case of Hamiltonian diffeomorphisms of tori. 
\begin{thm}\label{t:main}
Let $z$ be a symplectically degenerate maximum of a Hamiltonian diffeomorphism $\phi$ of the standard symplectic torus $(\T^{2d},\omega)$, with $d\geq1$. If, for all $n\in\N$, $z$ is isolated in the set of contractible fixed points of $\phi^n$, then the action of $z$ is a non-isolated point of the average-action spectrum of $\phi$.
\end{thm}

In the paper, we will actually prove a more precise statement (Theorem~\ref{t:symplectically_degenerate_maxima}) that also gives some information on the period of the periodic points that are found. For its proof, instead of relying on the variational principle associated with the Hamiltonian action functional, we make use of a ``discrete'' version of it, introduced by Chaperon \cite{Chaperon:Une_idee_du_type_geodsiques_brisees_pour_les_systmes_hamiltoniens, Chaperon:An_elementary_proof_of_the_Conley_Zehnder_theorem_in_symplectic_geometry}. The functional associated with this variational principle, called the discrete symplectic action,  is a generating family for the graph of the Hamiltonian diffeomorphism. 
In the final section of the paper, we will employ this function to provide quick proofs of special cases of the Conley conjecture: for Hamiltonian diffeomorphisms of $(\T^{2d},\omega)$ whose graph is described by a single generating function (which is almost a corollary of Theorem~\ref{t:main}), and for non-degenerate Hamiltonian diffeomorphisms of $(\T^{2d},\omega)$ (which is one of the celebrated Conley-Zehnder's theorems).

\subsection{Organization of the paper}
In Section~\ref{s:Preliminaries} we provide the reader with the necessary background. We introduce the variational principle associated with the discrete symplectic action and discuss the relation between its Morse indices at a critical point and the Maslov index of the corresponding periodic point of the considered Hamiltonian diffeomorphism. Moreover, we review the definition and the basic properties of local homology groups. Section~\ref{s:strongly_visible_points} is the core of this paper: we define and investigate the properties of symplectically degenerate maxima, and we prove Theorem~\ref{t:main}. Finally, in Section~\ref{s:Conley_conjecture} we provide quick proofs of some special cases of the Conley conjecture. 

\subsection{Acknowledgments} 
I wish to thank Alberto Abbondandolo, Viktor Ginzburg and Nancy Hingston  for encouraging and for useful remarks, and Dietmar Salamon for explaining to me some details of the relation between Maslov and Morse indices of generating families. I especially thank the anonymous referee for his careful reading of the preprint, for catching several inaccuracies in the first draft, and for providing important comments. This research was supported by the ANR project ``KAM faible''.


\section{Preliminaries}\label{s:Preliminaries}

Throughout this paper, $\phi$ will denote a Hamiltonian diffeomorphism of the standard symplectic torus $(\T^{2d},\omega)$, where $d\geq1$. For us, $\T^{2d}$ is the quotient $\R^{2d}/\Z^{2d}$ and we will denote the global Darboux coordinates on it as usual by $z=(x,y)$, so that $\omega=\diff x\wedge\diff y$. We refer the reader to \cite{McDuff_Salamon:Introduction_to_symplectic_topology, Hofer_Zehnder:Symplectic_invariants_and_Hamiltonian_dynamics} for an introduction to symplectic geometry and Hamiltonian dynamics.

\subsection{The discrete symplectic action}\label{s:discrete_symplectic_action}
We are interested in the periodic points of $\phi$, namely at those $z_0\in\T^{2d}$ such that $\phi^p(z_0)=z_0$ for some period $p\in\N=\{1,2,3,...\}$. By definition of Hamiltonian diffeomorphism, $\phi=\phi_1$ is the time-1 map of a Hamiltonian flow $\phi_t:\T^{2d}\to\T^{2d}$, and without loss of generality one can always assume $\phi_t$ to be generated by a Hamiltonian that is $1$-periodic in time, in such a way that 
\begin{align}
\label{e:Hamiltonian_periodic}
\phi_{t+1}=\phi_{t}\circ\phi_1. 
\end{align}
Therefore, there is a one-to-one correspondence between periodic points $z_0$ of $\phi$ and periodic orbits $\gamma(t)=\phi_t(z_0)$ of the corresponding Hamiltonian flow. A well-known property of Hamiltonian diffeomorphisms is that, at least on closed symplectically aspherical manifolds (among which are the tori), the fact that a periodic orbit corresponding to a periodic point of $\phi$ be contractible does not depend on the choice of the Hamiltonian flow $\phi_t$ having $\phi$ as time-1 map, see \cite[Prop.~3.1]{Schwarz:On_the_action_spectrum_for_closed_symplectically_aspherical_manifolds}. Hence, it makes sense to talk about \textbf{contractible periodic points} of the Hamiltonian diffeomorphism~$\phi$.

It is well known that contractible periodic orbits of Hamiltonian systems can be detected by means of a variational principle: they are critical points of the action functional. In the current paper, we will make use of a discrete version of this variational principle, which was introduced by Chaperon \cite{Chaperon:Une_idee_du_type_geodsiques_brisees_pour_les_systmes_hamiltoniens, Chaperon:An_elementary_proof_of_the_Conley_Zehnder_theorem_in_symplectic_geometry} in the early eighties in order to provide a simple proof of the Arnold conjecture for tori. Here, we recall this principle following Robbin and Salamon \cite{Robbin_Salamon:Phase_functions_and_path_integrals}. First of all, let us factorize $\phi$ as
\begin{align}\label{e:factorization_phi}
 \phi=\psi_{k-1}\circ\psi_{k-2}\circ...\circ\psi_0,
\end{align}
where each $\psi_j:\T^{2d}\to\T^{2d}$ is the Hamiltonian diffeomorphism defined by 
\[
\psi_j
:=
\phi_{(j+1)/k}\circ\phi_{j/k}^{-1}.
\]
By choosing the discretization step $k$ large enough we can make the $\psi_j$'s as $C^1$-close to the identity as we wish. Eventually, we can lift them to $\Z^{2d}$-equivariant Hamiltonian diffeomorphisms $\tilde\psi_j:\R^{2d}\to\R^{2d}$ that are still $C^1$-close to the identity, and therefore  admit generating functions $\tilde f_j:\R^{2d}\to\R$ such that
\begin{align}\label{e:generating_function}
  \begin{array}{c}
    \Big.(x_{j+1},y_{j+1})=\psi_j(x_j,y_j) \\ 
    \bigg.\mbox {if and only if} \\ 
    \left\{
    \begin{array}{@{}l}
    x_{j+1}-x_j=\partial_y \tilde f_j(x_{j+1},y_j),\Big. \\ 
    y_{j+1}-y_j=-\partial_x \tilde f_j(x_{j+1},y_j).\Big. 
  \end{array}
\right. \\ 
  \end{array}
\end{align}
Hereafter, $j$ must be interpreted as a cyclic index, i.e.\ $j\in\Z_k$. The functions $\tilde f_j$ are invariant by the action of $\Z^{2d}$, and therefore they descend to functions of the form 
\[f_j:\T^{2d}\to\R,\] 
see \cite[Lemma~11.1]{McDuff_Salamon:Introduction_to_symplectic_topology} or \cite[page~223]{Gole:Symplectic_twist_maps}. We define the \textbf{discrete symplectic action}
\[ 
\tilde\A:(\R^{2d})^{\times k}=\underbrace{\R^{2d}\times...\times\R^{2d}}_{\times k}\to\R
\]
by
\[
\tilde\A(\zz)
:=
\sum_{j\in\Z_k}
\Big(
\langle
y_j,x_j-x_{j+1}
\rangle
+
f_j(x_{j+1},y_j)
\Big).
\]
Here, $\zz=(z_0,...,z_{k-1})$ and $z_j=(x_j,y_j)$. Notice that $\tilde\A$ is invariant by the diagonal action of $\Z^{2d}$ on $(\R^{2d})^{\times k}$, and therefore it descends to a function 
\[\A:(\R^{2d})^{\times k}/\Z^{2d}\to\R.\] 
A straightforward computation shows that $[\zz]\in(\R^{2d})^{\times k}/\Z^{2d}$ is a critical point of $\A$ if and only if $z_{j+1}=\tilde\psi_j(z_j)$ for every $j\in\Z_k$, and therefore if and only if $[z_0]\in\T^{2d}$ is a contractible fixed point of the Hamiltonian diffeomorphism $\phi$. In particular, critical points of $\A$ are in one-to-one correspondence with contractible fixed points of the Hamiltonian diffeomorphism $\phi$. 

Let $H_t:\T^{2d}\to\R$ be the 1-periodic Hamiltonian function generating the flow $\phi_t$. Set $\zeta(t)=(\chi(t),\nu(t)):=\tilde\phi_t(z_0)$ to be an arbitrary  line of the lifted flow, $z_j:=\zeta(\tfrac jk)$, and $\zz=(z_0,...,z_{k-1})$. Up to normalizing each generating function $f_j$ with an additive constant (depending only on $H_t$), we have
\begin{align*}
f_j[(x_{j+1},y_j)]=\langle y_j,x_{j+1}-x_j\rangle + \int_{j/k}^{(j+1)/k} \Big(- \langle \nu(t),\tfrac{\diff}{\diff t}\chi(t)\rangle + H_t(\zeta(t)) \Big)\,\diff t,
\end{align*}
If $[(x_{j+1},y_j)]$ is a critical point of $f_j$,  its critical value is given by
\begin{align*}
\int_{j/k}^{(j+1)/k} \Big(- \langle \nu(t),\tfrac{\diff}{\diff t}\chi(t)\rangle + H_t(\zeta(t)) \Big)\,\diff t.
\end{align*}
Analogously, if $[\zz]$ is a critical point of $F$,  its critical value is 
\begin{align*}
\int_{0}^{1} \Big(- \langle \nu(t),\tfrac{\diff}{\diff t}\chi(t)\rangle + H_t(\zeta(t)) \Big)\,\diff t.
\end{align*}
This shows that the critical values of the discrete symplectic action $F$ are indeed action values of the periodic orbits corresponding to the critical points of $F$. We define the \textbf{action spectrum} of $\phi$ to be the set of critical values of $F$. In the following, in order to simplify the notation we will omit the squared brackets~$[\cdot]$.

We wish to apply the machinery of critical point theory, and more specifically Morse theory, to the discrete symplectic action. In order to do this, we need to make sure that the so-called Palais-Smale condition~\cite{Palais:Morse_theory_on_Hilbert_manifolds} is satisfied: every sequence $\{\zz_\alpha\ |\ \alpha\in\N\}$ such that $\A(\zz_\alpha)\to c$ and $|\diff\A(\zz_\alpha)|\to0$ admits a converging subsequence. 

\begin{prop}\label{p:Palais_Smale}
The discrete symplectic action $\A$ satisfies the Palais-Smale condition, and its critical points  are contained in a compact subset of its domain.
\end{prop}

\begin{proof}
Following McDuff and Salamon~\cite[page~352]{McDuff_Salamon:Introduction_to_symplectic_topology}, let 
\[
\tau:(\R^{2d})^{\times k}/\Z^{2d}\toup^{\cong}\T^{2d}\times(\R^{2d})^{\times k-1}\] 
be the diffeomorphism given by
\[
\tau(z_0,...,z_{k-1})=(z_0,z_1-z_0,z_2-z_1,...,z_{k-1}-z_{k-2}).
\]
We have
\begin{align*}
\A\circ\tau^{-1}(z_0,\zzeta)
=
\langle A\cchi,\uupsilon\rangle 
+
B(z_0,\zzeta),\s\forall z_0\in\T^{2d},\ \zzeta=(\zeta_0,...,\zeta_{k-2})\in(\R^{2d})^{\times k-1},
\end{align*}
where $B$ is a function of the form 
\[B:(\T^{2d})^{\times k}\to\R,\] 
$A$ is the $d(k-1)\times d(k-1)$ matrix which can be written as an upper triangular matrix with the upper part filled by blocks of $d\times d$ identity matrices $I$, i.e.
\[
A=
\left[
  \begin{array}{cccc}
     I & I &  \cdots &   I\\ 
     0 & I & \ddots &   \vdots\\ 
     \vdots & \ddots & \ddots &   I\\ 
     0 & \cdots & 0 &   I\\ 
  \end{array}
\right],
\]
and we write each $\zeta_j$ in symplectic coordinates as $(\chi_j,\upsilon_j)$, so that  $\cchi:=(\chi_0,...,\chi_{k-2})$ and $\uupsilon:=(\upsilon_0,...,\upsilon_{k-2})$ are vectors in $(\R^d)^{\times k-1}$. 

Now, the quadratic form $Q(\zzeta):=\langle A\cchi,\uupsilon\rangle$ is non-degenerate, and therefore $|\diff Q(\zzeta)|\to\infty$ as $|\zzeta|\to\infty$. On the other hand the function $B$ has bounded $C^1$ norm, being defined on a compact space. This proves the proposition.
\end{proof}

The variational principle associated with the discrete symplectic action is available in any period $p\in\N$. Namely, contractible $p$-periodic points of $\phi$ (that is, contractible fixed points of the $p$-fold composition $\phi^p=\phi\circ...\circ\phi$) correspond to critical points of the discrete symplectic action $\An{p}:(\R^{2d})^{\times kp}/\Z^{2d}\to\R$ defined by
\begin{align}\label{e:symplectic_action_period_n}
\An{p} (\zz)
:=
\sum_{j\in\Z_{kp}}
\Big(
\langle
y_j,x_j-x_{j+1}
\rangle
+
f_{j\, \mathrm{mod}\, k}(x_{j+1},y_j)
\Big), 
\end{align}
where $\zz=(z_0,...,z_{kp-1})$ and $z_j=(x_j,y_j)$. We define the \textbf{average-action spectrum} of $\phi$ to be the union over $p\in\N$ of the sets of critical values of $\tfrac1p \An p$, i.e.
\begin{align*}
\left\{
\tfrac1p\An p(\zz)\ |\ p\in\N,\ \zz\in\mathrm{crit}(\An p)
\right\}.
\end{align*}
This set coincides with the set of average-actions of the periodic orbits of the Hamiltonian flow~$\phi_t$.

\subsection{The Morse indices}\label{s:Morse_indices}

Detecting contractible periodic points of $\phi$ is equivalent to detecting  critical points of the functions $\An p$ (for all $p\in\N$). However, if $\zz$ is a critical point of $\A$,  its product
\[ \zz\iter n:=(\underbrace{\big.\zz,...,\zz}_{\times n}  ) \]
is a critical point of $\An n$ for every $n\in\N$. This is another way of saying that a fixed point of $\phi$ is also a fixed point of any iterated composition of $\phi$ with itself. In order to estimate the number of contractible periodic points of $\phi$ one needs some criteria to establish whether a given critical point $\zz$ of $\An p$ is a ``genuine'' one or it is rather of the form $\zz=\ww\iter n$ for some $n< p$ and some critical point $\ww$ of $\An{p/n}$. In certain situations, this can be achieved by looking at the Morse indices. 

Let $\zz$ be a critical point of $\An p$. We denote by $\mor(\zz)$ and $\nul(\zz)$ the Morse index and the nullity of $\An{p}$ at the critical point $\zz$. We recall that these are nonnegative integers defined respectively as the dimension of the negative eigenspace  and of the kernel of the Hessian of $\An{p}$ at $\zz$. The nullity of critical points of the discrete symplectic action admits a dynamical characterization according to the following proposition. We recall that $\phi$ factorizes as the composition $\psi_{k-1}\circ\psi_{k-2}\circ...\circ\psi_0$.

\begin{prop}\label{p:nullity} 
A vector $\ZZ=(Z_0,...,Z_{kp-1})\in(\R^{2d})^{\times kp}$ belongs to the kernel of the Hessian of $\An p$ at a critical point $\zz=(z_0,...,z_{kp-1})$ if and only if 
\[Z_{j+1}=\diff\psi_{j\, \mathrm{mod}\, k}(z_j)Z_j,\s\s \forall j\in\Z_{kp}.\] 
In particular 
\begin{align}\label{e:nullity}
\nul(\zz)=\dim\ker(\diff\phi(z_0)^p-I). 
\end{align}
\end{prop}
\begin{proof}
In order to simplify the notation, let us set $p=1$. If we write each $\psi_j$ as $(\chi_j,\upsilon_j)$, then by~\eqref{e:generating_function} the differential of $\psi_j$ can be written in matrix form as
\begin{align*}
\diff\psi_j
&=
\left[
  \begin{array}{cc}
    \partial_x\chi_j & \partial_y\chi_j \\ \noalign{\medskip}
    \partial_x\upsilon_j & \partial_y\upsilon_j  
  \end{array}
\right],
\end{align*}
and
\begin{align*}
 \partial_x\chi_j &= (I-\partial_{xy}f_j)^{-1}, \\
 \partial_y\chi_j &= (I-\partial_{xy}f_j)^{-1} \partial_{yy}f_j,\\
 \partial_x\upsilon_j &= -\partial_{xx}f_j (I-\partial_{xy}f_j)^{-1},\\
 \partial_y\upsilon_j &= I-\partial_{yx}f_j-\partial_{xx}f_j(I-\partial_{xy}f_j)^{-1}\partial_{yy}f_j.
\end{align*}
Given $Z_j=(X_j,Y_j)\in\R^{2d}$, the vector $Z_{j+1}=(X_{j+1},Y_{j+1}):=\diff\psi_{j}(z_j)Z_j$ is given by
\begin{small}
\begin{equation}\label{e:diff_psi}
\begin{split}
X_{j+1} &= (I-\partial_{xy}f_j)^{-1}( X_j +  \partial_{yy}f_j Y_j ),\\
Y_{j+1} &= -\partial_{xx}f_j (I-\partial_{xy}f_j)^{-1} X_j + (I-\partial_{yx}f_j-\partial_{xx}f_j(I-\partial_{xy}f_j)^{-1}\partial_{yy}f_j)Y_j, 
\end{split}
\end{equation} 
\end{small}%
where all the derivatives of $f_j$ are meant to be evaluated at the point $(x_{j+1},y_j)$. On the other hand, the Hessian of $\A$ at $\zz$ is given by
\begin{align*}
\hess\A(\zz)[\ZZ,\ZZ']
=
&
\sum_{j\in\Z_k}
\langle
X_j - (I-\partial_{xy}f_j) X_{j+1} + \partial_{yy}f_j Y_j
,
Y_j'
\rangle\\
&
+
\sum_{j\in\Z_k}
\langle
\partial_{xx}f_j X_{j+1} + Y_{j+1} - (I-\partial_{yx}f_j) Y_{j} 
,
X_{j+1}'
\rangle,
\end{align*}
for every $\ZZ=(Z_0,...,Z_{k-1})$ and $\ZZ'=(Z_0',...,Z_{k-1}')$, with $Z_j=(X_j,Y_j)$ and $Z_j'=(X_j',Y_j')$. Therefore, $\ZZ$ belongs to the kernel of the Hessian of $\A$ at $\zz$ if and only if~\eqref{e:diff_psi} holds for every $j\in\Z_k$. Notice that, for such $\ZZ=(Z_0,...,Z_{k-1})$, we have
\[
\diff\phi(z_0)Z_0=\diff\psi_{k-1}(z_{k-1})\diff\psi_{k-2}(z_{k-2})...\diff\psi_0(z_0)Z_0=Z_0.
\]
Hence, the map $\ZZ\mapsto Z_0$ is a linear isomorphism from the kernel of $\hess\A(\zz)$ to the kernel of $(\diff\phi(z_0)-I)$, and~\eqref{e:nullity} follows.
\end{proof}

In order to employ the Morse indices for distinguishing among critical points, it is useful to know how the sequences $\{\mor(\zz\iter n)\,|\,n\in\N\}$ and $\{\nul(\zz\iter n)\,|\,n\in\N\}$ behave. All we need to know about the nullity is described by  the following proposition. We recall that the \textbf{Floquet multipliers} of $\phi^p$ at a fixed point $z_0$ are defined as the eigenvalues of $\diff\phi(z_0)^p$.
\begin{prop}\label{p:same_nullity}
Let $\zz$ be a critical point of $\An p:(\R^{2d})^{\times kp}/\Z^{2d}\to\R$. For every integer $n>1$, we have that $\nul(\zz)=\nul(\zz\iter n)$ if and only if none of the Floquet multipliers of $\phi^{p}$ at $z_0$ other than 1 are  $n$-th complex roots of $1$.\hfill\qed
\end{prop}

\begin{proof}
By applying Long's version of Bott's formulas for the geometric multiplicity of eigenvectors of  symplectic matrices, we infer
\begin{align*}
\dim\ker(\diff\phi(z_0)^{pn}-I)=\sum_{\lambda\in\sqrt[n]1} \dim_\C\ker_\C(\diff\phi(z_0)^{p}-\lambda I),
\end{align*}
see \cite[Theorem~9.2.1]{Long:Index_theory_for_symplectic_paths_with_applications}. Hence, no Floquet multiplier of $\phi^{p}$ at $z_0$ other than 1 is  an $n$-th complex root of $1$ if and only if the above equation reduces to
\begin{align*}
\dim\ker(\diff\phi(z_0)^{pn}-I)= \dim_\C\ker_\C(\diff\phi(z_0)^{p}-I) = \dim\ker(\diff\phi(z_0)^{p}-I).
\end{align*}
By Proposition~\ref{p:nullity} our statement follows.
\end{proof}

As for the Morse index, one may directly develop an analog of Bott's iteration theory \cite{Bott:On_the_iteration_of_closed_geodesics_and_the_Sturm_intersection_theory}. However, we prefer to proceed indirectly by relying on the relation between the Morse index of the discrete symplectic action and the Maslov index of the periodic points of $\phi$, and then by using Long's extensions of Bott's theory for the Maslov index.

\subsection{The Maslov index}\label{s:Maslov_index}
The symplectic group $\Sp(2d)$ has an infinite cyclic fundamental group. The Maslov index is an integer $\mas(\Gamma)$ which is assigned to elements $\Gamma$ of the universal cover $\widetilde{\Sp}(2d)$. Here, we regard an element of $\widetilde{\Sp}(2d)$ covering $A\in\Sp(2d)$ as a homotopy class of continuous paths $\Gamma:[0,1]\to\Sp(2d)$ joining the identity $I$ with $A$. The integer $\mas(\Gamma)$ is roughly the number of half-windings made by the path $\Gamma$ in the symplectic group. In the following we briefly recall its precise definition, and we refer the reader to the books \cite{Abbondandolo:Morse_theory_for_Hamiltonian_systems, Long:Index_theory_for_symplectic_paths_with_applications} for a detailed account.

Let $r:\Sp(2d)\to U(d)$ be the retraction that sends any $A\in\Sp(2d)$ to the unitary complex matrix obtained from $(AA^T)^{-1/2}A$ after identifying $\R^{2d}$ with $\C^{d}$. By composing this retraction with the determinant we obtain the so-called rotation function 
\[\rho:=\det\circ\, r:\Sp(2d)\to S^1\subset\C.\] 
Given $\Gamma\in\widetilde\Sp(2d)$, let $\theta:[0,1]\to\R$ be a continuous function such that $e^{i2\pi\theta(t)}=\rho(\Gamma(t))$ for all $t\in[0,1]$. We define the \textbf{average Maslov index} $\avmas(\Gamma)$ by 
\[\avmas(\Gamma):=2(\theta(1)-\theta(0)).\] 
The reason for this terminology lies in the property described later in~\eqref{e:iteration_long}. Notice that the average Maslov index is not necessarily an integer. 

Now, we denote by $\Sp^*(2d)\subset\Sp(2d)$ the subset of those symplectic matrices that do not have $1$ as an eigenvalue. Its complement $\Sp^0(2d):=\Sp(2d)\setminus\Sp^*(2d)$ is a singular hypersurface in the symplectic group that separates $\Sp^*(2d)$ in two connected components. Now, consider the $2d\times2d$ symplectic diagonal matrices $W':=\diag(2,1/2,-1,-1,...,-1)$ and $W'':=-I=\diag(-1,-1,...,-1)$. These two matrices belong to different connected components of $\Sp^*(2d)$, and we have $\rho(W')=(-1)^{d-1}=-\rho(W'')$. Given $\Gamma\in\widetilde\Sp(2d)$ with $\Gamma(1)\in\Sp^*(2d)$, let us choose an arbitrary continuous path $\Gamma':[0,1]\to\Sp^*(2d)$  joining $\Gamma(1)$ with either $W'$ or $W''$. We denote by $\Gamma*\Gamma'\in\widetilde\Sp(2d)$ the path obtained by concatenating  $\Gamma$ and $\Gamma'$. Then, the \textbf{Maslov index} of $\Gamma$ is defined as
\[
\mas(\Gamma):=\avmas(\Gamma*\Gamma').
\]
The function $\mas:\widetilde\Sp(2d)\to\Z$ is locally constant on the set of those $\Gamma\in\widetilde\Sp(2d)$ with $\Gamma(1)\in\Sp^*(2d)$, and it cannot be continuously extended to the whole $\widetilde\Sp(2d)$. We extend it as a lower semi-continuous function by setting\footnote{For other applications, different extensions may be more suitable. For instance, Robbin and Salamon consider in~\cite{Robbin_Salamon:The_Maslov_index_for_paths} the average between the maximal lower semi-continuous and the minimal upper semi-continuous extensions.}
\begin{align}\label{e:Maslov_degenerate}
\mas(\Gamma):=\liminf_{
\scriptsize
  \begin{array}{c}
    \Psi\to\Gamma \\ 
    \Psi(1)\!\in\!\Sp^*\!(2d) \\ 
  \end{array}
}\mas(\Psi).
\end{align}

Now, for each $n\in\N$, we denote by $\Gamma\iter n\in\widetilde{\Sp}(2d)$ the $n$-fold product of $\Gamma$ in the universal cover group $\widetilde{\Sp}(2d)$, i.e.
\[
\Gamma\iter n(\sfrac{j+t}{n})=\Gamma(t)\underbrace{\Gamma(1)...\Gamma(1)}_{\times j},
\s\s\forall j\in\{0,...,n-1\},\ t\in[0,1].
\]
The Maslov index $\mas(\Gamma\iter n)$ grows according to the following iteration formula established by Liu and Long \cite{Liu_Long:An_optimal_increasing_estimate_of_the_iterated_Maslov-type_indices, Liu_Long:Iteration_inequalities_of_the_Maslov-type_index_theory_with_applications}.
\begin{align}\label{e:iteration_long}
 n\,\avmas(\Gamma)-d\leq \mas(\Gamma\iter n)\leq n\,\avmas(\Gamma)+d-\dim\ker(\Gamma\iter n(1)-I).
\end{align}
Moreover, if these inequalities are not both strict, then $1$ is the only eigenvalue of $\Gamma\iter n(1)$. If $1$ is the only eigenvalue of $\Gamma(1)$ and $\avmas(\Gamma)=0$, then $\mas(\Gamma)=\mas(\Gamma\iter n)$ for all $n\in\N$. We refer the reader to \cite[chapters 9,10]{Long:Index_theory_for_symplectic_paths_with_applications} for a proof of these (and many others) iteration properties.

Now, let us consider again the Hamiltonian diffeomorphism $\phi:\T^{2d}\to\T^{2d}$ and a contractible $p$-periodic point $z_0$ of it. Let $\phi_t$ be a Hamiltonian flow  whose time-1 map is $\phi$ and that is generated by a Hamiltonian that is 1-periodic in time. Consider the path $Z:[0,\infty)\to\Sp(2d)$ given by $Z(t)=\diff\phi_t(z_0)$. Let $\zz$ be the critical point of $\An p$ corresponding to the contractible fixed point $z_0$. We define the \textbf{average Maslov index} of $\zz$ and, for every period $n\in\N$, the \textbf{Maslov index} of $\zz\iter n$ as the integers
\begin{align*}
\avmas(\zz)&:=\avmas(Z|_{[0,p]}),\\
\mas(\zz\iter n)&:=\mas(Z|_{[0,pn]}),
\end{align*}
where the restrictions $Z|_{[0,p]}$ and $Z|_{[0,pn]}$ are seen as elements of $\widetilde{\Sp}(2d)$. It turns out that the Maslov index of $\zz$ depends only on $z_0$ and $\phi$, and not on the specific choice of the Hamiltonian flow (this is indeed true for all closed symplectically aspherical manifolds, see \cite[page~440]{Schwarz:On_the_action_spectrum_for_closed_symplectically_aspherical_manifolds}). By equation~\eqref{e:Hamiltonian_periodic}, we have that $Z(t+1)=Z(t)Z(1)$. The following proposition summarizes  the above-mentioned iteration properties of the Maslov index.
\begin{prop}\label{p:iteration_properties_Long}
Let $\zz=(z_0,...,z_{kp-1})$ be a critical point of $\An p$. 
\begin{itemize}
\item[(i)] For all $n\in\N$ we have
\begin{align*}
 n\,\avmas(\zz)-d\leq \mas(\zz\iter n)\leq n\,\avmas(\zz)+d-\nul(\zz\iter n).
\end{align*}
Moreover,  if these inequalities are not both strict, then $1$ is the only Floquet multiplier of $\phi^{pn}$ at $z_0$. 

\item[(ii)] If $1$ is the only Floquet multiplier of $\phi^{p}$ at $z_0$ and $\avmas(\zz)=0$, then $\mas(\zz)=\mas(\zz\iter n)$ for all $n\in\N$. \hfill\qed
\end{itemize}
\end{prop}

The Maslov index is related to the Morse index of the discrete symplectic action according to the following statement, which may be regarded as a symplectic analog of the Morse index Theorem.

\begin{prop}[Symplectic Morse index Theorem]
\label{p:Symplectic_Morse_index_Theorem}
For every critical point $\zz$ of $\An {p}$ we have
\begin{align}\label{e:Maslov_Morse}
\mor(\zz)=\mas(\zz) + \underbrace{\sfrac12\dim(\mathrm{domain}(\An {p}))}_{=dkp}.
\end{align}
\end{prop}
\begin{proof}
Assume first that the critical point $\zz$ of $\An{p}$ is non-degenerate. By Proposition~\ref{p:nullity} this is equivalent to the fact that $z_0$ is a non-degenerate fixed point of the $p$-fold composition $\phi\circ...\circ\phi$. In this case, equation~\eqref{e:Maslov_Morse} was proved by Robbin and Salamon in \cite[Theorem~4.1]{Robbin_Salamon:Phase_functions_and_path_integrals}. 

Now, consider the Morse index as a  locally constant function on the space of non-degenerate symmetric bilinear forms. Its extension to the whole space of symmetric bilinear forms is the maximal lower semi-continuous one. In the same way, in equation~\eqref{e:Maslov_degenerate} we defined the Maslov index on the whole $\widetilde{\Sp}(2d)$ as the maximal lower semi-continuous extension of the Maslov index for paths with final endpoint in $\Sp^*(2d)$ (i.e.\ the ``non-degenerate'' paths). Hence equation \eqref{e:Maslov_Morse} still holds when $\zz$ is a degenerate critical point of $\An p$.
\end{proof}

\subsection{Local homology}\label{s:local_homology} 
Let $M$ be a smooth manifold, and $p$ an isolated  critical point of a smooth function $f:M\to\R$. If $p$ is non-degenerate, by the Morse Lemma~\cite[page~6]{Milnor:Morse_theory} we know that on a neighborhood of $p$ the function $f$ looks like a non-degenerate quadratic form whose negative eigenspace has dimension equal to the Morse index $\mor(p)$. If $p$ is degenerate, i.e.\ its nullity $\nul(p)$ is non-zero, then the behavior of $f$ around $p$ can be wilder, as described by the Generalized Morse Lemma~\cite[Lemma~1]{Gromoll_Meyer:On_differentiable_functions_with_isolated_critical_points}. Some properties of the critical point $p$ are captured by its \textbf{local homology}, which is defined as the graded group
\[
\Loc_*(p):=\Hom_*(\{f<c\}\cup\{p\},\{f<c\}),
\]
where $c=f(p)$ and $\Hom_*$ denotes the singular homology functor. For our purpose, it is enough to consider  homology with coefficients in the field $\Z_2$. We refer the reader to \cite[Chapter~1]{Chang:Infinite_dimensional_Morse_theory_and_multiple_solution_problems} for the main properties of local homology groups. Here, we wish to recall that $\Loc_j(p)$ is always trivial in degrees $j<\mor(p)$ and $j>\mor(p)+\nul(p)$. Moreover, if a strip $\{c_1\leq f\leq c_2\}$ is ``sufficiently compact'' (for instance if the  Palais-Smale condition holds with respect to a complete Riemannian metric) and contains only isolated critical points of $f$, then we have the generalized Morse inequality
\begin{align}\label{e:Morse_inequality}
\rank \Hom_j(\{f<c_2\},\{f<c_1\})
\leq
\sum_p
\rank \Loc_j(p), 
\end{align}
where the above sum runs over all the critical points $p$ of $f$ with $c_1\leq f(p)<c_2$. 

In the next section, we will need the following properties of local homology groups. We divide them into two statements: the first one is local whereas the second one is global.

\begin{prop}\label{p:maximal_degree_local_homology}
Assume that an isolated critical point $p$ of a smooth function $f:M\to\R$ has non-trivial local homology (with coefficients in $\Z_2$) in maximal degree $k=\mor(p)+\nul(p)$.
Then the following properties hold.
\begin{itemize}
\item[(i)] The local homology of $p$ is concentrated in degree $k$ and it is isomorphic to the coefficient field, i.e. $\Loc_*(p)=\Loc_{k}(p) \cong \Z_2$.

\item[(ii)] For every smoothly embedded $k$-dimensional disc $D\subset M$ containing  $p$ in its interior and such that $f|_{D\setminus\{p\}}<f(p)$, the group $\Loc_*(p)$ is generated by $[D]$.

\item[(iii)] Let $\theta_t$ be the anti-gradient flow of $f$ with respect to any Riemannian metric. There exist arbitrarily small open neighborhoods $U\subset M$ of $p$ such that 
\begin{itemize}
\item[$\bullet$] $U$ has the mean-value property: if $\theta_{t_1}(p')$ and $\theta_{t_2}(p')$ belong to $U$ for some $p'\in\R^N$ and $t_1<t_2$, then $\theta_{t}(p')$ belongs to $U$ for all $t\in[t_1,t_2]$;

\item[$\bullet$] $\partial U$ is the union of three smooth compact manifolds with boundary:  the entry set
\begin{align*}
U_{\entry}&=
\{
p'\in\overline U\ |\ \theta_t(p')\not\in \overline U\ \ \forall t<0
\}
\end{align*}
that is contained in a superlevel set $\{f\geq f(p)+\delta\}$ for some $\delta>0$, the exit set
\begin{align*}
U_{\exit}
&=
\{
p'\in\overline U\ |\ \theta_t(p')\not\in \overline U\ \ \forall t>0
\}
\end{align*}
that is contained in the sublevel set $\{f< f(p)\}$, and $U_{\tan}$ that is tangent to the anti-gradient flow $\theta_t$;

\item[$\bullet$] the inclusion induces a homology isomorphism
\[
\Loc_*(p)\toup^{\cong}\Hom_*(\{f<f(p)\}\cup U,\{f<f(p)\}).
\]
\end{itemize}
\end{itemize}
\end{prop}

\begin{proof}
Since the statement is local we can assume that $M=\R^N$ and $p$ is the origin. Let $m:=\mor(p)$ and $n:=\nul(p)$, so that $k=m+n$. By the Generalized Morse Lemma we can assume that the function $f$ has the form
\begin{equation*}
f(x_0,x_-,x_+)=f_0(x_0)-|x_-|^2+|x_+|^2,\s\s
\forall x_0\in\R^n,\ x_-\in\R^m,\ x_+\in\R^{N-k}, 
\end{equation*}
where  $f_0:\R^n\to\R$ is a function whose Hessian vanishes at the origin. Let us denote by $\Loc^0_*(p)$  the local homology of $f_0$ at the origin. By Gromoll-Meyer's Shifting Theorem (see \cite{Gromoll_Meyer:On_differentiable_functions_with_isolated_critical_points} or  \cite[page~50]{Chang:Infinite_dimensional_Morse_theory_and_multiple_solution_problems}) we have that
$\Loc_*(p)
\cong 
\Loc_{*-m}^0(p)$. In particular $\Loc_{n}^0(p)\cong\Loc_{k}(p)$ is non-trivial, which in turn implies that the origin is an isolated local maximum for the function $f_0$. However, the local homology of isolated local maxima is concentrated in maximal degree $n$, where it is isomorphic to the coefficient field $\Z_2$ (see \cite[page~51]{Chang:Infinite_dimensional_Morse_theory_and_multiple_solution_problems}). This proves~(i).

Now, consider a disc $D$ as in the statement of the lemma, and set
\begin{align*}
\E^+ & :=\big\{(x_0,x_-,x_+)\in\R^N\ \big|\ x_0=0,\ x_-=0\big\}. 
\end{align*}
We claim that $D$ is transverse to $\E^+$ at the origin $p$. Indeed, assume that there is a non-zero $v\in\Tan_p D\cap\E^+$. Consider a smooth curve $\gamma:(-\epsilon,\epsilon)\to D$ such that $\dot\gamma(0)=v$. Then, we have
\begin{align*}
\left.\frac{\diff^2}{\diff t^2}\right|_{t=0} f\circ\gamma(t)
=
\left.\frac{\diff}{\diff t}\right|_{t=0} \diff f(\gamma(t))\dot\gamma(t)
=
\underbrace{\hess f(p)[v,v]}_{>0} + \underbrace{\diff f(p)\ddot\gamma(0)}_{=0}
>0,
\end{align*}
which contradicts the fact that $f<0$ on $D\setminus\{p\}$.

We set 
\[B_*(r):=\big\{ (x_0,x_-)\in\R^k\ \big|\  |x_0|^2+|x_-|^2\leq r^2 \big\}.\] 
By the Implicit Function Theorem, there exist a radius $r>0$ and a smooth map $\psi:B_*(r)\to \R^{N-k}$ such that the disc
\[
D':=\{(x_0,x_-,\psi(x_0,x_-))\ |\ (x_0,x_-)\in B_*(r)\}
\]
is a compact neighborhood of the origin $p$ in $D$. Notice that $[D']=[D]$ in $\Loc_k(p)$.  Let $h_t:B_*(r)\hookrightarrow \R^N$ (for $0\leq t\leq1$) be the isotopy given by 
\[h_t(x_0,x_-)=(x_0,x_-,(1-t)\psi(x_0,x_-)).\] 
Notice that $D':=h_0(B_*(r))$, and $D'':=h_1(D')$ is a generator of the local homology  $\Loc_k(p)$. Moreover, $f$ is a Lyapunov function for $h_t$, i.e.\ $\tfrac{\diff}{\diff t} f\circ h_t\leq0$. Hence $[D'']=[D']$ in $\Loc_k(p)$. This proves point~(ii).

As for point (iii), for the sake of simplicity let us choose the Euclidean Riemannian metric on the Euclidean space $M=\R^N$ (the general case being analogous), and let $\theta_t$ be the corresponding anti-gradient flow of $f$. We consider the function $f_*:\R^k\to\R$ given by
\[f_*(x_0,x_-)=f_0(x_0)-|x_-|^2,\s\s\forall (x_0,x_-)\in\R^k.\]
Choose an arbitrarily small  $\epsilon>0$ so that the origin is a global maximum and the only critical point of the function $f_*$ restricted to the closed ball $B_*(\epsilon)$. Set $c:=f_*(0)=f(p)$. We can find $\delta_1\in(0,\epsilon^2)$ arbitrarily small so that the space 
\[K:=B_*(\epsilon)\cap\{f_*\geq c-\delta_1\}\] 
is a compact neighborhood of the origin contained in the interior of the closed ball $B_*(\epsilon)$, with smooth boundary 
$\partial K = B_*(\epsilon)\cap\{f_*=c-\delta_1\}$. 
Let $\theta^*_t$ be the anti-gradient flow of $f_*$. Fix $\tau>0$ small enough so that $\theta^*_t(\partial K)$ is contained in the interior of $B_*(\epsilon)$ for all $t\in[0,\tau]$. Since the origin is the only rest point of the flow $\theta^*_t$, there exists $\delta_2\in(\delta_1,\epsilon^2)$ such that
\[
\theta^*_\tau(\partial K)\subset\{f_*<c-\delta_2\}.
\]
For $r>0$, we denote by $B_+(r)$ the closed ball of radius $r$ in $\R^{N-k}$, i.e. 
\[B_+(r):=\big\{ x_+\in\R^{N-k}\ \big|\  |x_+|^2\leq r^2 \big\}.\]
We define the closed set 
\[U_{\entry}:= K \times \partial B_+(\textstyle\sqrt{\delta_2}).\] 
Notice that
\begin{align*}
f(x_0,x_-,x_+)=f_*(x_0,x_-) + \delta_2 \geq c-\delta_1 + \delta_2 = c + \delta,\s\s
\forall (x_0,x_-,x_+)\in U_{\entry},
\end{align*}
where $\delta:=\delta_2-\delta_1>0$. The anti-gradient flow $\theta_t$ has the form 
\[\theta_t(x_0,x_-,x_+)=(\theta^*_t(x_0,x_-),\theta^+_t(x_+)),\] 
where $\theta^+_t$ is the anti-gradient flow of the quadratic form $f-f_*$. Thus, there exist a unique $\delta_3\in(0,\delta_2)$ such that
$\theta_\tau\big(\partial K\times\partial B_+(\sqrt{\delta_2})\big)
=
\theta^*_\tau(\partial K)\times\partial B_+(\sqrt{\delta_3})$. We define the closed sets
\begin{align*}
U_{\exit}&:=\theta^*_\tau(\partial K)\times B_+(\textstyle\sqrt{\delta_3}),\\
U_{\tan}&:=\bigcup_{t\in[0,\tau]}\theta_t\left(\partial K\times\partial B_+(\textstyle\sqrt{\delta_2})\right).
\end{align*}
Notice that $U_{\entry}$, $U_{\exit}$, and $U_{\tan}$ are smooth compact manifolds with boundary. The first two are transverse to the anti-gradient of $f$, while the third one is tangent to it. The piecewise smooth hypersurface  
$V:=U_{\entry}\cup U_{\exit}\cup U_{\tan}$  separates $\R^N$. We define $U$ to be the (relatively compact) connected component of the complement of $V$ containing the origin $p$. The set $U$ is contained in the interior of $B_*(\epsilon)\times B_+(\epsilon)$, has clearly the mean-value property, and its closure $\overline{U}$ has entry set $U_{\entry}$ and exit set $U_{\exit}$. Moreover
\begin{align*}
f(x_0,x_-,x_+)= f_*(x_0,x_-) + \delta_3 < c-\delta_2 + \delta_3< c ,\s\s
\forall (x_0,x_-,x_+)\in U_{\exit}.
\end{align*}
Let $\rho:\R^N\to[0,1]$ be a smooth function supported inside $B_*(\epsilon)\times B_+(\epsilon)$ and such that $\rho|_U\equiv 1$. For $t\in[0,1]$, we define a smooth homotopy 
\[k_t:(\{f<c\}\cup U,\{f<c\})\to(\{f<c\}\cup U,\{f<c\})\] 
by $k_t(x_0,x_-,x_+)=(x_0,x_-,(1-t\rho(x_0,x_-,x_+))x_+)$. Notice that $k_0$ is the identity, the function $f$ is non-increasing along the homotopy $k_t$, and we have
\begin{align*}
f\circ k_1(x_0,x_-,x_+)=f(x_0,x_-,0)=f_*(x_0,x_-),\s\s \forall (x_0,x_-,x_+)\in U.
\end{align*}
In particular, $k_1(U)\subset\{f<c\}\cup\{p\}$. Therefore, the map $k_1$ is a homotopy equivalence
\[k_1:(\{f<c\}\cup U,\{f<c\})\to(\{f<c\}\cup\{p\},\{f<c\})\] 
whose homotopy inverse is given by the inclusion. This implies that the inclusion induces a homology isomorphism between $\Loc_*(p)$ and $\Hom_*(\{f<c\}\cup U,\{f<c\})$.
\end{proof}

\begin{prop}\label{p:injection_of_local_homology__general_statement}
Let $(M,g)$ be a complete Riemannian manifold, and $f:M\to\R$ a smooth function satisfying the Palais-Smale condition. Assume that $f$ has an isolated critical point $p$ with non-trivial local homology (with coefficients in $\Z_2$) in maximal degree $\mor(p)+\nul(p)$. Set $c:=f(p)$. If, for some $b>c$, the interval $(c,b]$ does not contain critical values of $f$, then the inclusion induces a homology monomorphism
\[
\Loc_*(p)\hookrightarrow \Hom_*(\{f<b\},\{f<c\}).
\]
\end{prop}

\begin{rem}
If the level set $\{f=c\}$ contains only finitely many critical points of $f$, the statement follows from \cite[Theorem~3.2]{Chang:Infinite_dimensional_Morse_theory_and_multiple_solution_problems} even without any assumption on the local homology of $p$. \hfill\qed
\end{rem}

\begin{proof}[Proof of Proposition~\ref{p:injection_of_local_homology__general_statement}]
Notice that the critical point $p$ satisfies the assumptions of Prop\-o\-si\-tion~\ref{p:maximal_degree_local_homology}. Choose a suitable sufficiently small open neighborhood $U$ of the critical point $p$ satisfying the properties stated in Proposition~\ref{p:maximal_degree_local_homology}(iii). Thus, there exists $\delta\in(0,b-c)$ such that the entry set $U_{\entry}$ is contained in the superlevel set $\{f\geq c+\delta\}$. Now, let $C$ be the compact set of critical points of $f$ with critical value $c$, and set $C':=C\setminus\{p\}$. Choose an open  neighborhood $V'$ of $C'$ that is relatively compact in the sublevel set $\{f<c+\tfrac\delta2\}$ and such that $\overline{V'}\cap \overline U=\varnothing$. We denote by  $\theta_t$ the anti-gradient flow of $f$, and we define the open set
\[
V:=\bigcup_{t\in[0,+\infty)} \theta_t(V').
\]
By construction, $V$ is positively invariant under the anti-gradient flow $\theta_t$. We claim that $V\cap U=\varnothing$. Otherwise there exists $q\in V'$ and $t>0$ such that $\theta_t(q)\in U_{\entry}$, but this is impossible since $f(\theta_t(q))\leq f(q)<c+\tfrac\delta2$ whereas $U_{\entry}\subset\{f\geq c+\delta\}$. By replacing $U$ with a smaller open neighborhood of $p$ still satisfying the properties of  Proposition~\ref{p:maximal_degree_local_homology}(iii), we can actually achieve $\overline V\cap \overline U=\varnothing$.

By the properties of the neighborhoods $U$ (see Proposition~\ref{p:maximal_degree_local_homology}(iii)) and $V$, the open set $\{f<c\}\cup V\cup U$ is positively invariant under the anti-gradient flow $\theta_t$. The closed set $\{f\leq b\} \setminus \big(\{f<c\}\cup V\cup U\big)$ does not contain critical points of $\An{n}$. Therefore, we can use the anti-gradient flow $\theta_t$ to deform the pair   $(\{f< b\},\{f<c\})$ into $(\{f< c\}\cup V\cup U,\{f<c\})$, which implies that the inclusion induces a homology isomorphism
\begin{align*}
\iota'_*:\Hom_*(\{f<c\}\cup V\cup U,\{f<c\})
\toup^{\cong}
\Hom_*(\{f< b\},\{f< c\}). 
\end{align*}
Since $\overline U\cap\overline V=\varnothing$, by excision we have the isomorphism
\begin{small}
\begin{align*}
\Hom_*(\{f<c\}\cup V\cup U,\{f< c\})
\cong
\Hom_*(\{f< c\}\cup V,\{f< c\})
\oplus
\Hom_*(\{f< c\}\cup U,\{f< c\}).
\end{align*} 
\end{small}%
By Proposition~\ref{p:maximal_degree_local_homology}(iii), the inclusion induces a homology isomorphism
\begin{align*}
\Loc_*(p)\toup^{\cong}\Hom_*(\{f<c\}\cup U,\{f<c\}),
\end{align*}
and thus it induces a monomorphism
\begin{align*}
\iota''_*:\Loc_*(p)\hookrightarrow\Hom_*(\{f<c\}\cup V\cup U,\{f<c\}).
\end{align*}
By the functoriality of singular homology, $\iota'_*\circ\iota''_*=(\iota'\circ\iota'')_*$, and therefore the inclusion induces a homology monomorphism as claimed.
\end{proof}

\section{Symplectically degenerate maxima}\label{s:strongly_visible_points}

\subsection{Definition and basic properties}\label{s:def_sympl_deg_max}
Let us adopt the notation of Section~\ref{s:Preliminaries} concerning the Hamiltonian diffeomorphism $\phi=\psi_{k-1}\circ...\circ\psi_0$. Given a contractible fixed point $z_0$ of   $\phi$, it is well known that one can choose the factorization $\psi_{k-1}\circ...\circ\psi_0$ in such a way that $z_0$ is a fixed point of all the $\psi_j$'s (see \cite[section~9]{Salamon_Zehnder:Morse_theory_for_periodic_solutions_of_Hamiltonian_systems_and_the_Maslov_index} or \cite[section~5.1]{Ginzburg:The_Conley_conjecture}). Indeed one can choose a Hamiltonian flow $\phi_t$ such that $\phi_0=\mathrm{id}$, $\phi_1=\phi$ and $\phi_t(z_0)=z_0$ for all $t\in[0,1]$ in the following way. Let $\theta_t$ be a Hamiltonian flow satisfying $\theta_0=\mathrm{id}$, $\theta_1=\phi$ and $\theta_{t+1}=\theta_t\circ\theta_1$ for all $t\in\R$. Since the 1-periodic orbit $\gamma(t):=\theta_t(z_0)$ is contractible, there exists a smooth homotopy $\gamma_s:\R/\Z\to \T^{2d}$ such that $\gamma_0\equiv z_0=\gamma_s(0)$ for all $s\in[0,1]$, and $\gamma_1=\gamma$. For each $r\in(0,1/2]$ and $z\in\T^{2d}$ we denote by $B_r(z)$ the open ball of radius $r$ with respect to the flat Riemannian metric on $\T^{2d}$ (we recall that our torus is 1-periodic, i.e. $\T^{2d}=\R^{2d}/\Z^{2d}$). By a slight abuse of notation, for each $z'\in B_{1/2}(z)$, we will write $z'-z$ for the unique shortest vector $v\in\R^{2d}$ such that $z+v=z'$. For $s,t\in[0,1]$, we consider a smooth family of functions $\rho_{s,t}:\T^{2d}\to[0,1]$ such that $\rho_{s,t}\equiv 1$ on $B_{1/8}(\gamma_s(t))$ and $\rho\equiv0$ outside $B_{1/4}(\gamma_s(t))$. We define a smooth family of Hamiltonian functions $K_{s,t}:\T^{2d}\to\R$ supported inside $B_{1/4}(\gamma_s(t))$ by
\[
K_{s,t}(z):=\rho_{s,t}(z)\,\langle J\tfrac{\diff}{\diff s}\gamma_s(t),z-\gamma_{s}(t) \rangle,\s\s \forall z\in B_{1/4}(\gamma_s(t)).
\]
Let $\kappa_{s,t}$ be the family of Hamiltonian diffeomorphisms of $\T^{2d}$ defined by 
\begin{align*}
\kappa_{0,t}&=\mathrm{id},\\
\tfrac{\diff}{\diff s}\kappa_{s,t}&=-J\,\nabla K_{s,t}\circ \kappa_{s,t}.
\end{align*}
A straightforward computation shows that $\kappa_{s,t}(z_0)=\kappa_{s,t}(\gamma_s(0))=\gamma_s(t)$. Moreover $K_{s,0}=K_{s,1}\equiv0$, which implies $\kappa_{s,0}=\kappa_{s,1}=\mathrm{id}$. Thus, we can obtain the Hamiltonian flow $\phi_t$ with the desired properties by setting $\phi_t:=\kappa_{1,t}^{-1}\circ\theta_t$.

We say that $z_0$ is a \textbf{symplectically degenerate maximum}  when
\begin{itemize}
\setlength{\itemindent}{10pt}
\item[\textbf{(SDM1)}] $z_0$ is an isolated local maximum of the generating functions $f_0,...,f_{k-1}$,
\item[\textbf{(SDM2)}] the local homology $\Loc_{dkn+d}(z_0\iter{kn})$ is non-trivial for infinitely many $n\in\N$.
\end{itemize}
We denote by $\K_{z_0}\subset\N$ the infinite subset of those $n$ for which \textbf{(SDM2)} holds.

Assumption \textbf{(SDM2)} imposes strong constraints on the Morse indices and the local homology of symplectically degenerate maxima, according to the following statement.

\begin{prop}\label{p:local_homology_top_deg}
Let $\zz=(z_0,...,z_{k-1})$ be a critical point of $\A$ such that,  for infinitely many $n\in\N$, the point $\zz\iter n$ is  isolated in the set of critical points of $\An n$ and the local homology $\Loc_{dkn+d}(\zz\iter{n})$ is non-trivial. We denote by $\K_{\zz}\subset\N$ the infinite subset of those $n$ for which this  condition holds. Then, for all $n\in\K_{\zz}$, we have
\begin{itemize}
\item[(i)] $\mor(\zz\iter{n})+\nul(\zz\iter{n})=dkn+d$,
\item[(ii)] $1$ is the only Floquet multiplier of $\phi^n$ at $z_0$,
\item[(iii)] the local homology of $\zz\iter{n}$ is isomorphic to $\Z_2$ and concentrated in maximal degree $dkn+d$, i.e.\ $\Loc_{*}(\zz\iter{n})\cong\Loc_{dkn+d}(\zz\iter{n})\cong\Z_2$,
\item[(iv)] the multiples of $n$ are in $\K_{\zz}$,
\item[(v)] if $m$ divides $n$ and $\nul(\zz\iter m)=\nul(\zz\iter n)$, then $m\in\K_{\zz}$.
\end{itemize}
Moreover, 
\begin{itemize}
\item[(vi)] if $1$ is the only Floquet multiplier of $\phi$ at $z_0$, then $\K_{\zz}=\N$.
\end{itemize}
\end{prop}

\begin{proof}
The assumption on the local homology implies that
\begin{align*}
\mor(\zz\iter{n})
\leq
dkn+d
\leq
\mor(\zz\iter{n})
+
\nul(\zz\iter{n}),\s\s\forall n\in\K_{\zz}.
\end{align*}
By the Symplectic Morse index Theorem (Proposition~\ref{p:Symplectic_Morse_index_Theorem}) we infer
\begin{align*}
\mas(\zz\iter{n})\leq d\leq\mas(\zz\iter{n})+\nul(\zz\iter{n}),\s\s\forall n\in\K_{\zz}.
\end{align*}
In particular $\avmas(\zz)=0$.  This, together with Proposition~\ref{p:iteration_properties_Long}(i), implies
\[\mas(\zz\iter{n})+\nul(\zz\iter{n})\leq d,\s\s\forall n\in\N,\]
and thus
\[\mas(\zz\iter{n})+\nul(\zz\iter{n})=d,\s\s\forall n\in\K_{\zz},\] 
which is equivalent to~(i). Proposition~\ref{p:iteration_properties_Long}(i) further implies (ii). Point~(i), together with  Proposition~\ref{p:maximal_degree_local_homology}, implies (iii). 

Now, let us assume that $\nul(\zz\iter m)=\nul(\zz\iter n)$ for some positive integer $m$ that divides $n\in\K_{\zz}$. By point~(ii) and Prosposition~\ref{p:same_nullity}, the only Floquet multiplier of $\phi^m$ at $z_0$ is $1$. By Proposition~\ref{p:iteration_properties_Long}(ii), we infer that  
\begin{align}\label{e:inside_main_proof_mas_n_equal_mas_m}
\mas(\zz\iter m)=\mas(\zz\iter n).
\end{align}
Let $N_m\subset(\R^{2d})\iter {km}/\Z^{2d}$ be a central manifold of the vector field $\nabla \An m$ at $\zz\iter m$. We recall that $N_m$ is an invariant submanifold for the gradient flow of $\An m$ such that $\Tan_{\zz\iter m}N_m$ is the kernel of the Hessian of $\An m$ at $\zz\iter m$, see \cite[Section~3.2]{Guckenheimer_Holmes:Nonlinear_oscillations_dynamical_systems_and_bifurcations_of_vector_fields}. In particular $\dim(N_m)=\nul(\zz\iter m)$. Consider the diagonal embedding \[\itmap {n/m}:(\R^{2d})\iter{km}/\Z^{2d}\hookrightarrow(\R^{2d})\iter{kn}/\Z^{2d}\] given by
$ \itmap {n/m}(\ww)=\ww\iter {n/m}$. 
A straightforward calculation shows that \[\nabla\An n\circ \itmap{n/m}=\itmap{n/m}\circ\nabla \An m,\] where the gradients in this equation are with respect to the standard flat metrics of $(\R^{2d})\iter {km}/\Z^{2d}$ and $(\R^{2d})\iter {kn}/\Z^{2d}$. This implies that the  manifold $N_n:=\itmap{n/m}(N_m)$ is invariant under the gradient flow of $\An n$. Moreover, 
\[
\dim(N_n)=\dim(N_m)=\nul(\zz\iter m)=\nul(\zz\iter n),
\]
and therefore $\Tan_{\zz\iter n}N_n$ is the kernel of the Hessian of $\An n$ at $\zz\iter n$. This proves that $N_n$ is a central manifold of the vector field $\nabla \An n$ at $\zz\iter n$.
Set $c:=\A(\zz)$, so that $\An m(\zz\iter m)=mc$ and $\An n(\zz\iter n)=nc$. We have
\begin{equation}\label{e:loc_homology_grom_meyer_proof_main}
\begin{split}
\Loc_{j+\mor(\zz\iter n)}(\zz\iter n) & \cong \Hom_{j}(\{\An n|_{N_n}<nc\}\cup \{\zz\iter n\}, \{\An n|_{N_n}<nc\})\\
& \cong\Hom_{j}(\{\An m|_{N_m}<mc\}\cup \{\zz\iter m\}, \{\An m|_{N_m}<mc\})\\
& \cong \Loc_{j+\mor(\zz\iter m)}(\zz\iter m),
\end{split}
\end{equation}
where the first and third isomorphisms follow from Gromoll-Meyer's splitting Theorem (see \cite[Section~3]{Gromoll_Meyer:On_differentiable_functions_with_isolated_critical_points}). By \eqref{e:inside_main_proof_mas_n_equal_mas_m} and Proposition~\ref{p:Symplectic_Morse_index_Theorem} we have that 
\[\mor(\zz\iter n)=\mas(\zz\iter n)+dkn=\mas(\zz\iter m)+dkn=\mor(\zz\iter m)+dk(n-m).\]
This, together with~\eqref{e:loc_homology_grom_meyer_proof_main}, implies that $\Loc_{dkm+d}(\zz\iter m)\cong\Loc_{dkn+d}(\zz\iter n)$, and proves point~(v). The proof of~(iv) is analogous: if $m$ is now a multiple of $n\in\K_{\zz}$, then $1$ is the only Floquet multiplier of $\phi^m$ at $z_0$, which implies that $\nul(\zz\iter n)=\nul(\zz\iter m)$ and $\Loc_{dkn+d}(\zz\iter n)\cong\Loc_{dkm+d}(\zz\iter m)$. Finally, (vi) follows from (iv--v) and Proposition~\ref{p:same_nullity}.
\end{proof}

\subsection{An example}
It is easy to construct a Hamiltonian diffeomorphism of a  torus with a symplectically degenerate maximum. Consider a $C^2$-small smooth function $f_0:\T^{2d}\to\R$ having an isolated totally degenerate local maximum at $z_0$ with critical value $f_0(z_0)=0$. By totally degenerate we mean that the Hessian of $f_0$ at $z_0$ vanishes. Let $\phi$ be the Hamiltonian diffeomorphism of  $(\T^{2d},\omega)$ having $f_0$ as generating function. Notice that $z_0$ is a contractible fixed point of $\phi$. We claim that $z_0$ is a symplectically degenerate maximum. Indeed, consider the discrete symplectic action $\An n:(\R^{2d})\iter n/\Z^{2d}\to\R$ given by
\[
\An n(\zz)
=
\sum_{j\in\Z_n}
\Big(\langle y_j,x_j-x_{j+1}\rangle
+
f_0(x_{j+1},y_j)
\Big).
\]
Since $z_0$ is a totally degenerate maximum of $f_0$, we have that
\[
\hess\An n(z_0\iter n)[\bm{Z'},\bm{Z''}]
=
\sum_{j\in\Z_n}
\Big(
\langle Y_j',X_j''-X_{j+1}''\rangle
+
\langle Y_j'',X_j'-X_{j+1}'\rangle
\Big),
\]
for all $\bm{Z'}=(Z_0',...,Z_{n-1}'),\bm{Z''}=(Z_0'',...,Z_{n-1}'') \in(\R^{2d})\iter n$. An easy computation shows that $\mor(z_0\iter n)=dn-d$ and $\nul(z_0\iter n)=2d$. Consider the submanifold
\[
N_n=\left\{ w\iter n\ |\ w\in\T^{2d}\right\}\subset(\R^{2d})\iter n/\Z^{2d}.
\]
It is easy to see that this is a central manifold for the gradient of $\An n$ at $z_0\iter n$. Namely $N_n$ is a submanifold invariant by the gradient flow of $\An n$, and whose tangent space at $z_0\iter n$ is the kernel of $\hess\An n(z_0\iter n)$. By Gromoll-Meyer's splitting Theorem (see \cite[Section~3]{Gromoll_Meyer:On_differentiable_functions_with_isolated_critical_points})  we have
\begin{align*}
\Loc_j(z_0\iter n)
&\cong
\Hom_{j-\mor(z_0\iter n)}(\{\An n|_{N_n}<0\}\cup\{z_0\iter n\},\{\An n|_{N_n}<0\})\\
&\cong
\Hom_{j-dn+d}(\{f_0<0\}\cup\{z_0\},\{f_0<0\})\\
&=\Loc_{j-dn+d}(z_0).
\end{align*}
In particular $\Loc_{dn+d}(z_0\iter n)\cong\Loc_{2d}(z_0)$ and, since $z_0$ is an isolated local maximum of $f_0$, this latter local homology group is non-trivial. This shows that $z_0$ is a symplectically degenerate maximum.

\subsection{Average-action spectrum and symplectically degenerate maxima}

In this subsection we prove Theorem~\ref{t:main} stated in the introduction. The result is a consequence of the following homological vanishing property, which is inspired by an analogous statement due to Bangert and Klingenberg \cite[Theorem~2]{Bangert_Klingenberg:Homology_generated_by_iterated_closed_geodesics} in the setting of closed geodesics.

\begin{lem}[Homological vanishing]
\label{l:homological_vanishing}
Let $z_0$ be a  symplectically degenerate maximum of the Hamiltonian diffeomorphism $\phi=\psi_{k-1}\circ...\circ\psi_0:\T^{2d}\to\T^{2d}$. Assume that $z_0$ is isolated in the set of contractible fixed points of $\phi^n$ for all $n\in\K_{z_0}$. Set $c:=F(z_0\iter k)$ and choose an arbitrary $\epsilon>0$.  Then, the inclusion-induced homomorphism
\begin{align*}
\iota_*:\Loc_*(z_0\iter {kn})\to\Hom_*(\{\An{n}<nc+\epsilon\},\{\An{n}<nc\})
\end{align*}
is trivial, provided $n\in\K_{z_0}$ is large enough.
\end{lem}

\begin{proof}
In order to simplify the notation, we will work inside a fundamental domain of the universal cover of $(\R^{2d})\iter {kn}/\Z^{2d}$. Hence we can assume that the domain of $\An n$ is simply $(\R^{2d})\iter {kn}$ without taking the quotient by $\Z^{2d}$. Moreover, let us assume without loss of generality that $z_0=0$ and $c=f_0(0)=...=f_{k-1}(0)=0$. Let $R>0$ be sufficiently small so that, for all $j\in\Z_k$, $f_j<0$ outside the origin  in a ball of radius $3R$ centered at $0$. For $n\in\K_n$, consider the $(dkn+d)$-dimensional vector subspace
\begin{align*}
\E_n=
\left\{\zz=(x_0,y_0,...,x_{kn-1},y_{kn-1})\ \Bigg|\ 
  \begin{array}{l}
    x_0,...,x_{kn-1},w \in\R^d \\ 
    y_j=w+x_{j+1}-x_j \ \ \forall j\in\Z_{kn}  
  \end{array}
\right\},
\end{align*}
and, for $0<r<R$, the polydisc
\begin{gather*}
W_n=W_n(R,r)=
\Big\{\zz\in \E_n\ \Big|\ |x_0|\leq R,...,|x_{kn-1}|\leq R, |w|\leq r\Big\}.
\end{gather*}
On $W_n$, the discrete symplectic action  $\An n$ is given by
\begin{gather*}
\An n(\zz)=
\sum_{j\in\Z_{kn}}
\Big(
-|x_j-x_{j+1}|^2
+
f_j(x_{j+1},w+x_{j+1}-x_j )
\Big),
\s\s\forall \zz\in W_n.
\end{gather*}
By Propositions~\ref{p:maximal_degree_local_homology} and~\ref{p:local_homology_top_deg}(iii),  $[W_n]$ is a generator of the local homology of $\An n$ at $0\iter n$. 

Now, fix a vector $v\in\R^d$ such that $|v|=r$,  and let $h:[0,1]\times W_n\to(\R^{2d})\iter{n}$ be the homotopy given by
\begin{align*}
h(t,\zz)=\zz + t\zz', 
\end{align*}
where $\zz=(x_0,y_0,...,x_{kn-1},y_{kn-1})$, $\zz'=(0,y_0',0,y_1',...,0,y_{kn-1}')$, and 
\begin{align*}
&y_j'=0,\s\s\forall j\in\{0,...,kn'-1\},\\ 
&y_j'=v,\s\s\forall j\in\{kn',...,kn-1\}. 
\end{align*}
Notice that
\begin{equation}\label{e:estimate_bangert_klingenberg_lemma}
\begin{split}
\An n\circ h(t,\zz)= &
\sum_{j=0}^{kn'-1}
\Big(
-|x_j-x_{j+1}|^2
+
f_j(x_{j+1},w+x_{j+1}-x_j )
\Big)\\
&+ \sum_{j=kn'}^{kn-1}
\Big(
-|x_j-x_{j+1}|^2+
f_j(x_{j+1},tv+w+x_{j+1}-x_j )\Big)\\
& +t\,\langle v,x_{kn'}-x_0\rangle.
\end{split}
\end{equation}
Let $\rho:[0,R]\to[0,\infty)$ be a monotone increasing continuous function such that
\[ \rho(s)\leq \min\left\{s^2,\min_
{ \scriptsize
  \begin{array}{c}
    |z''|=s, \\ 
    j\in\Z_{k} \\ 
  \end{array}
}
\left\{-f_j(z'')\right\}\right\}. \]
This function allows us to express the following estimate:
\[
-|x_j-x_{j+1}|^2
+
f_j(x_{j+1},u+x_{j+1}-x_j )
\leq
-\rho\big(\tfrac{|u|}{2}\big),\s\s\forall u\in\R^d\mbox{ with }|u|\leq R.
\]
Now, let us apply this estimate to~\eqref{e:estimate_bangert_klingenberg_lemma}.  Notice that, since $|v|=r\geq|w|$, at least one of the vectors $w$ and $v+w$ has norm   larger than or equal to $r/2$. Hence, if we choose $r<\epsilon/(2R)$, $n'>2Rr/\big(\rho\big(\tfrac{r}{4}\big)k\big)$ and $n>2n'$, we have
\begin{align*}
\An n\circ h(t,\zz)\leq &
\,t\,\langle v,x_{kn'}-x_0\rangle
\leq
t\,|v|\,|x_{kn'}-x_0|\leq 2Rr<\epsilon,
\\
\An n\circ h(1,\zz)\leq &
-kn'\rho\big(\tfrac{r}{4}\big)  +\underbrace{\langle v,x_{kn'}-x_0\rangle}_{\leq |v|(|x_{kn'}|+|x_0|)}\leq -kn'\rho\big(\tfrac{r}{4}\big)+2Rr<0.
\end{align*}
This proves that the homotopy $h$ deforms the disc $W_n$ into the action sublevel set $\{\An n<0\}$, and the disc $W_n$ remains in the  action sublevel set $\{\An n<\epsilon\}$ along the homotopy. In order to conclude the proof it is enough to show that, up to further assuming  $r<\rho(R)/2R$, the boundary $\partial W_n$ of the disc remains in the  action sublevel set $\{\An n<0\}$ along the homotopy. Consider an arbitrary $\zz\in\partial W_n$. We have two (non-exclusive) cases: $|w|=r$ or $|x_j|=R$ for some $j\in\Z_{kn}$. In the former case, $|w|=r$, we have
\begin{align*}
\An n\circ h(t,\zz)&\leq 
-kn'\rho\big(\tfrac{|w|}{2}\big)  +t\,\langle v,x_{kn'}-x_0\rangle\\
&\leq  -kn'\rho\big(\tfrac{r}{2}\big)+2Rr\\
&\leq -kn'\rho\big(\tfrac{r}{4}\big)+2Rr\\
&< 0.
\end{align*}
In the latter case, $|x_j|=R$ for some $j\in\Z_{kn}$, we have
\[
\An n\circ h(t,\zz)\leq 
\underbrace{f_{j-1}(x_{j},y_{j-1}+ty_{j-1}')}_{\leq-\rho(R)}
+t\,\underbrace{\langle v,x_{kn'}-x_0\rangle}_{\leq 2Rr}
\leq
-\rho(R)+2Rr<0.\qedhere
\]
\end{proof}

Lemma~\ref{l:homological_vanishing}  implies abundance of periodic points ``localized'' around the average-action of symplectically degenerate maxima in the average-action spectrum of $\phi$. More precisely, we have the following theorem, from which Theorem~\ref{t:main} directly follows.

\begin{thm}
\label{t:symplectically_degenerate_maxima}
Let $z_0$ be a symplectically degenerate maximum of a Hamiltonian diffeomorphism $\phi=\psi_{k-1}\circ...\circ\psi_0$ of the standard symplectic torus $(\T^{2d},\omega)$. Assume that $z_0$ is isolated in the  set of contractible fixed points of $\phi^n$ for all $n\in\K_{z_0}$. Then, there exist $n_{z_0}\in\K_{z_0}$ and a sequence $\{z_n\,|\,n\in\K_{z_0},\, n\geq n_{z_0}\}$ of contractible periodic points of $\phi$ with the following properties: $z_n$ is $n$-periodic and, if we denote by  $\zz_n$  the critical point of $\An{n}$ corresponding to $z_n$,  
we have $\An{n}(\zz_n)>\An{n}(z_0\iter{k n})$ and
\begin{align}\label{e:converging_average_action}
\lim_{n\to\infty} \Big(\An{n}(\zz_n)- \An{n}(z_0\iter{k n}) \Big)=0. 
\end{align}
\end{thm}

\begin{rem}\label{r:growth_rate}
Equation \eqref{e:converging_average_action} implies that $\{z_n\,|\,n\in\K_{z_0},\, n\geq n_{z_0}\}$ is an infinite subset of contractible periodic points belonging to pairwise distinct orbits of $\phi$. However,  each periodic point $z_n$ may have basic period that is a proper divisor of $n$. Nevertheless, if $\K_{z_0}$ contains all but finitely many prime numbers (for instance if 1 is the only Floquet multiplier of $\phi$ at $z_0$, according to Proposition~\ref{p:local_homology_top_deg}(vi)),  the growth rate of the contractible periodic orbits of $\phi$  having average-action in an arbitrarily small neighborhood of $F(z_0^{k})$  is at least the one of the prime numbers. Indeed, if $F(z_0^{k})$ is an isolated critical value of $F$, for every sufficiently large prime number $n$ contained in $\K_{z_0}$, the periodic point $z_n$ has minimal period $n$. \hfill\qed
\end{rem}

\begin{rem}\label{r:powers_of_prime}
Another interesting remark is that,  if $F(z_0^{k})$ is an isolated critical value of $F$ and $\K_{z_0}$ contains infinitely many powers of a prime number $p$, Theorem~\ref{t:symplectically_degenerate_maxima} implies the existence of infinitely many contractible periodic points belonging to pairwise distinct orbits of $\phi$, having basic periods in the set of powers of $p$ and average-action in any arbitrarily small neighborhood of $F(z_0^{k})$. \hfill\qed
\end{rem}

\begin{proof}[Proof of Theorem \ref{t:symplectically_degenerate_maxima}]
By contradiction, let us assume that such a sequence of periodic points does not exist. Hence, there exists $\epsilon>0$ and an infinite subset $\K\subset\K_{z_0}$ such that, for every $n\in\K$, there are no critical values of $\An n$ in the interval $(nc,nc+\epsilon]$, where $c:=\A(z_0\iter k)$. Notice that the critical point $z_0\iter{kn}$ of $\An{n}$ satisfies the assumptions of Proposition~\ref{p:injection_of_local_homology__general_statement}. This latter proposition implies that the inclusion induces a monomorphism
\begin{align*}
\iota_*:\Loc_*(z_0\iter{kn})\hookrightarrow\Hom_*(\{\An{n}<nc+\epsilon\},\{\An{n}<nc\}).
\end{align*}
This gives us a contradiction. Indeed, the local homology $\Loc_*(z_0\iter{kn})$ is non-trivial for all $n\in\K$, and the homological vanishing property of Lemma~\ref{l:homological_vanishing} implies that $\iota_*$ must be trivial provided $n$ is large enough.
\end{proof}

\subsection{Symplectically degenerate minima}

The reader may wonder if there is a  notion of symplectically degenerate minimum. Such a notion exists and is  analogous to the one of symplectically degenerate maximum, see Hein \cite{Hein:The_Conley_conjecture_for_the_cotangent_bundle}. Briefly, we can say that symplectically degenerate minima arise when working  with superlevel sets instead of sublevel sets. In this section, we give a definition using generating functions, and we state the main properties corresponding to the statements proved in the previous subsections for symplectically degenerate maxima. The proofs,  \emph{mutatis mutandis}, are entirely analogous and are left to the reader.

Let $\zz$ be a critical point of the discrete symplectic action $\An n$, with  $c:=\An n(\zz)$. We denote by $\mor^+(\zz):=2dkn-\mor(\zz)-\nul(\zz)$ the Morse co-index of $\An n$ at $\zz$, and we define the inverted local homology by
\[
\Loc_*^+(\zz):=\Hom_*(\{\An n>c\}\cup\{\zz\},\{\An n>c\}).
\]
We say that $z_0\in\T^{2d}$ is a \textbf{symplectically degenerate minimum}  when
\begin{itemize}
\setlength{\itemindent}{10pt}
\item[\textbf{(SDm1)}] $z_0$ is an isolated local minimum of the generating functions $f_0,...,f_{k-1}$,
\item[\textbf{(SDm2)}] the inverted local homology $\Loc_{dkn+d}^+(z_0\iter{kn})$ is non-trivial for infinitely many $n\in\N$.
\end{itemize}
We denote by $\K_{z_0}\subset\N$ the infinite subset of those $n$ for which \textbf{(SDm2)} holds. The easiest example of symplectically degenerate minimum $z_0$ is when $k=1$ and $z_0$ is an totally degenerate local minimum of the generating function $f_0$. The homological vanishing property of Lemma~\ref{l:homological_vanishing} holds for symplectically degenerate minima after replacing the local homology with the inverted local homology, and all the sublevel sets with superlevel sets. Theorem~\ref{t:symplectically_degenerate_maxima} can be modified as follows.

\begin{thm}
\label{t:symplectically_degenerate_minima}
Let $z_0$ be a symplectically degenerate minimum of a Hamiltonian diffeomorphism $\phi=\psi_{k-1}\circ...\circ\psi_0$ of the standard symplectic torus $(\T^{2d},\omega)$. Assume that $z_0$ is isolated in the  set of contractible fixed points of $\phi^n$ for all $n\in\K_{z_0}$. Then, there exist $n_{z_0}\in\K_{z_0}$ and a sequence $\{z_n\,|\,n\in\K_{z_0},\, n\geq n_{z_0}\}$ of contractible periodic points of $\phi$ with the following properties: $z_n$ is $n$-periodic and, if we denote by  $\zz_n$  the critical point of $\An{n}$ corresponding to $z_n$,  
we have $\An{n}(\zz_n)<\An{n}(z_0\iter{k n})$ and

\vspace{2pt}

$\displaystyle \s\s\s\s\s\s\s\s\s\s\lim_{n\to\infty} \Big(\An{n}(z_0\iter{k n}) - \An{n}(\zz_n) \Big)=0.$ \hfill\qed
\end{thm}

\section{The Conley conjecture}\label{s:Conley_conjecture}

Theorem~\ref{t:symplectically_degenerate_maxima} allows us to provide an easy proof of the Conley conjecture for the special case of Hamiltonian diffeomorphisms of standard symplectic tori that can be described by a single generating function (this latter condition being always satisfied for Hamiltonian diffeomorphisms that are sufficiently $C^1$-close to the identity).

\begin{thm}\label{t:conley}
Let $\phi$ be a Hamiltonian diffeomorphism of the standard symplectic torus $(\T^{2d},\omega)$ that is described by a single generating function. If the set of contractible fixed points of $\phi$ is finite, then $\phi$ has a contractible periodic point of basic period $p$ for all prime numbers $p$ large enough.
\end{thm}

\begin{proof}
Let $\phi$ be a Hamiltonian diffeomorphism satisfying the assumptions of the theorem. In particular, adopting the notation of section~\ref{s:Preliminaries}, we can take $k=1$ in the factorization~\eqref{e:factorization_phi}. Let us assume that the theorem does not hold for $\phi$: thus $\phi$ must have only finitely many contractible fixed points and, for all  $p$ belonging to an infinite set $\K$ of prime numbers, all the contractible $p$-periodic points of $\phi$ must be fixed points.  In order to get a contradiction and complete the proof, it is enough to show that $\phi$ admits a symplectically degenerate maximum $z_0$ with $\K_{z_0}=\N$, and then apply Theorem~\ref{t:symplectically_degenerate_maxima} and Remark~\ref{r:growth_rate}.

Let $\F\subset  \C$ be the (finite) subset of the Floquet multipliers of the contractible fixed points of $\phi$, and  $\F':=\{q\in \Q\,|\, e^{i2\pi q}\in \F\}$. We denote by $p_0$ the maximum among the denominators of the rational numbers in $\F'$, and we set $\K':=\K\cap(p_0,\infty)$. By Proposition~\ref{p:same_nullity}, for every contractible fixed point $z_0$ of $\phi$, we have
\begin{align}\label{e:same_nullity_in_proof_main_thm}
 \nul(z_0)=\nul(z_0^{\times n}),\s\s\forall n\in\K'.
\end{align}

Consider an arbitrary $n\in\K'$. By our assumptions on $\phi$, the discrete symplectic action $\An p:(\R^{2d})\iter n/\Z^{2d}\to\R$ has only finitely many critical points. Fix $c>0$ large enough so that all the critical values of $\An n$ are contained in the interval $(-c,c)$. Following McDuff and Salamon as in the proof of Proposition~\ref{p:Palais_Smale}, after a suitable change of coordinates we can assume that the domain of the function $\An n$ is $\T^{2d}\times(\R^{2d})^{\times n-1}$, and
\begin{align*}
\An n(z_0,\zzeta)=Q(\zzeta) + B(z_0,\zzeta), 
\s\s\forall z_0\in\T^{2d},\ \zzeta\in(\R^{2d})^{\times n-1}. 
\end{align*}
Here $Q(\zzeta)=\langle Q'\zzeta,\zzeta\rangle$ is a non-degenerate quadratic form with Morse index $dn-d$, while $B$ is a function of the form $B:(\T^{2d})^{\times n}\to\R$. Let us consider the orthogonal spectral decomposition associated with the symmetric matrix $Q'$, i.e.
\[
(\R^{2d})^{\times n-1}=\E^+\oplus \E^-,
\]
where $\E^+$ [resp.\ $\E^-$] is the positive [resp.\ negative] eigenspace of $Q'$. Since the Morse index of $Q$ is $dn-d$, we have
\begin{align*}
\dim(\E^+)=\dim(\E^-)=dn-d.
\end{align*}
For every $r>0$ we define
\begin{align*}
N(r)&:=\big\{ (z_0,\zzeta)\in\T^{2d}\times(\R^{2d})^{\times n-1}\ \big|\ Q(\zzeta)\leq r \big\},\\
L(r)&:=\big\{ (z_0,\zzeta)\in N(r)\ \big|\ Q(\zzeta)\leq- r \big\}.
\end{align*}
By the K\"unneth formula and excision,  we have
\begin{align*}
\Hom_j(N(r),L(r))&\cong\bigoplus_{i\in\Z}\Hom_{j-i}(\T^{2d})\otimes\Hom_i(\{Q\leq r\},\{Q\leq -r\})\\
&\cong\bigoplus_{i\in\Z}\Hom_{j-i}(\T^{2d})\otimes\Hom_i(D^{dn-d},\partial D^{dn-d})\\
&\cong \Hom_{j-(dn-d)}(\T^{2d}), 
\end{align*}
where $D^{dn-d}$ is the unit disc in $\E^-$. In particular
\begin{align}\label{e:homology_isolating_block}
\Hom_{dn+d}(N(r),L(r))\cong\Hom_{2d}(\T^{2d})\neq 0.
\end{align}
For $r>0$ large enough, the set $N(r)$ contains all the critical points of $\An n$, $L(r)$ contains no critical points of $\An n$, and  the gradient of $\An n$ is transverse to the boundaries of $N(r)$ and $L(r)$. Hence, we can use the gradient flow of $\An n$ to deform  $N(r)$ onto the sublevel set $\{\An n\leq c\}$ and $L(r)$ onto the sublevel set $\{\An n\leq -c\}$ respectively. This, together with~\eqref{e:homology_isolating_block}, implies
\[
\Hom_{dn+d}(\{\An n<c\},\{\An n<-c\})
\cong
\Hom_{dn+d}(N(r),L(r))\neq 0.
\]
By the Morse inequality (see~\eqref{e:Morse_inequality}) we infer that there exists a critical point $\zz(n)$ of $\An n$ such that 
\begin{align}\label{e:local_homology_Conley_Zehnder}
\Loc_{dn+d}(\zz(n))\neq0. 
\end{align}

By the  assumptions on $\phi$ that we made in the first paragraph of the proof, for every $n\in\K'$,  the periodic point of $\phi$ corresponding to $\zz(n)$ is a fixed point. Since $\phi$ has only finitely many contractible fixed points, one of them, say $z_0$, must correspond to infinitely many critical points in the sequence $\{\zz(n)\ |\ n\in\K'\}$. Namely, we can find an infinite subset $\K''\subset\K'$ such that $\zz(n)=z_0^{\times n}$ for all $n\in\K''$. Thus, we have found a contractible fixed point $z_0$ satisfying assumption \textbf{(SDM2)} in the definition of symplectically degenerate maximum (see section~\ref{s:def_sympl_deg_max}).

Now, fix an arbitrary $n\in\K''$. By Proposition~\ref{p:local_homology_top_deg}(ii), the only Floquet multiplier of $\phi^n$ at $z_0$ is~$1$. Thus the Floquet multipliers of $\phi$ at $z_0$ are all complex $n$th roots of $1$. By~\eqref{e:same_nullity_in_proof_main_thm} and Proposition~\ref{p:same_nullity}, the only Floquet multiplier of $\phi$ at $z_0$ is $1$ as well. This, together with Proposition~\ref{p:local_homology_top_deg}(vi),  implies that the local homology $\Loc_{2d}(z_0)$ is non-trivial, which is equivalent to the fact that $z_0$ is an isolated local maximum of the generating function $\An 1=f_0:\T^{2d}\to\R$. Therefore $z_0$ satisfies \textbf{(SDM1)}, and it is a symplectically degenerate maximum with $\K_{z_0}=\N$.
\end{proof}

\begin{rem}
Theorem~\ref{t:conley} implies that the growth rate of the periodic orbits of $\phi$ is at least like the one of prime numbers. Namely, if $N(p)$ is the number of contractible periodic orbits of $\phi$ of basic period less than or equal to $p$, we have

\vspace{5pt}

$\displaystyle \s\s\s\s\s\s\s\s\s\s\s\s\s\s\liminf_{p\to\infty} N(p) \frac{\log p}{p}\geq1.$ \hfill\qed
\end{rem}

If we relax the assumption that $\phi$ is described by a single generating function, the proof of Theorem~\ref{t:conley} goes through almost entirely, except for the last paragraph. There, the proof would still allow to conclude that the fixed point $z_0$ has an associated critical point $\zz=(z_0,...,z_{k-1})$ of $\An 1$ with non-trivial local homology $\Loc_{dk+d}(\zz)$. However, since  $k$ is now larger than 1, the discrete symplectic action in period $1$ is not merely the generating function $f_0$, it is rather a function of the form $\An 1:(\R^{2d})\iter k/\Z^{2d}\to\R$. Thus, in order to conclude, one must apply an argument due to Hingston \cite[section~4]{Hingston:Subharmonic_solutions_of_Hamiltonian_equations_on_tori} (see also \cite[section~5]{Ginzburg:The_Conley_conjecture}) showing that, in this situation, one can replace the factorization of $\phi$ with another equivalent one having the property that $z_0$ is an isolated local maximum of all the generating functions of the new factors.

As it was already noticed by Conley and Zehnder in one of their celebrated papers \cite{Conley_Zehnder:Subharmonic_solutions_and_Morse_theory}, the situation is significantly simpler if one considers Hamiltonian diffeomorphisms whose periodic points are all non-degenerate. Indeed, under this assumption, none of the periodic points satisfies condition \textbf{(SDM2)}. We conclude the paper by providing a proof of  Conley-Zehnder's Theorem under the slightly weaker assumption that the contractible fixed points are non-degenerate. 

\begin{thm}
Let $\phi$ be a Hamiltonian diffeomorphism of the standard symplectic torus $(\T^{2d},\omega)$ all of whose contractible fixed points are non-degenerate (i.e.\ none of their Floquet multipliers is equal to $1$). Then, for all prime numbers $p$ large enough, $\phi$ has at least two contractible periodic points of basic period $p$ and belonging to different orbits.
\end{thm}

\begin{proof}
Let us adopt the notation of section~\ref{s:discrete_symplectic_action}: in particular we consider a suitable factorization of $\phi$ as in~\eqref{e:factorization_phi} and, for every period $n\in\N$, we consider the associated discrete symplectic action $\An n:(\R^{2d})^{\times kn}/\Z^{2d}\to\R$ defined in~\eqref{e:symplectic_action_period_n}. By our assumption and Proposition~\ref{p:nullity}, all the critical points of $\An 1$ are non-degenerate. In particular, they are all isolated critical points, and since they are contained in a compact subset of the domain of $(\R^{2d})^{\times k}/\Z^{2d}$ (see Proposition~\ref{p:Palais_Smale}) they are finitely many.

Let $\F\subset  \C$ be the (finite) subset of the Floquet multipliers of the contractible fixed points of $\phi$, and  $\F':=\{q\in \Q\,|\, e^{i2\pi q}\in \F\}$. We denote by $p_0$ the maximum among the denominators of the rational numbers in $\F'$, and by $\K'$ the set of all prime numbers larger than $p_0$. By Proposition~\ref{p:same_nullity}, for any critical point $\zz$ of $\An 1$, we have
\begin{align}\label{e:Conley_Zehnder_nul_zero}
 \nul(\zz^{\times n})=\nul(\zz)=0,\s\s\forall n\in\K'.
\end{align}
By Propositions~\ref{p:iteration_properties_Long}(i) and~\ref{p:Symplectic_Morse_index_Theorem}, if $\avmas(\zz)=0$ we have that $-d<\mas(\zz\iter n)<d$ and thus $dkn-d<\mor(\zz\iter n)<dkn+d$ for all $n\in\K'$. If $\avmas(\zz)\neq0$, for all $n\in\K'$ larger than some $n_{\zz}$ we have that $|\mas(\zz\iter n)|>d$, and either $\mor(\zz\iter n)<dkn-d$ or $\mor(\zz\iter n)>dkn+d$. We denote by $n_0$ the maximum among the $n_{\zz}$ (for all critical points $\zz$ of $\An 1$), and by $\K$ the set $\K'\cap[n_0,\infty)$ of all prime numbers larger than $p_0$ and $n_0$. Summing up, for all critical points $\zz$ of $\An 1$ we have
\begin{align}\label{e:Conley_Zehnder_mas_good}
 \mor(\zz^{\times n})\not=dkn-d\s\mbox{and}\s\mor(\zz^{\times n})\not=dkn+d,\s\s\forall n\in\K.
\end{align}

Now, fix an arbitrary $n\in\K$ and choose $c>0$ large enough so that all the critical values of $\An n$ are contained in the interval $(-c,c)$. With the exact same argument as in the proof of Theorem~\ref{t:conley}, we have
\[
\Hom_{*}(\{\An n<c\},\{\An n<-c\})
\cong
\Hom_{*-(dkn-d)}(\T^{2d}).
\]
In particular $\Hom_{j}(\{\An n<c\},\{\An n<-c\})$ is non-trivial for $j=dkn-d$ and $j=dkn+d$. Let us assume that the discrete symplectic action $\An n$ has only finitely many critical points (otherwise $\phi$ has infinitely many contractible periodic points with basic period $n$, and we are done). By the generalized Morse inequality~\eqref{e:Morse_inequality}, $\An n$ has critical points $\zz'$ and $\zz''$ such that $\Loc_{dkn-d}(\zz')$ and $\Loc_{dkn+d}(\zz'')$ are non-trivial. 

Notice that the contractible periodic points corresponding to  $\zz'$ and $\zz''$ are not fixed points, and thus have basic period $n$. Indeed, by~\eqref{e:Conley_Zehnder_mas_good}, every critical point of $\An n$ of the form $\zz\iter n$ has Morse index different than $dkn\pm d$. By~\eqref{e:Conley_Zehnder_nul_zero}, $\zz\iter n$ is a non-degenerate critical point, and thus its local homology is non-trivial only in degree $\mor(\zz\iter n)$.

If we write $\zz'=(z_0',...,z_{kn-1}')$, we have that $z_{k(j+1)}'=\phi(z_{kj}')$ for all $j\in\Z_{n}$. Notice that there is an action of $\Z_{n}$ on the domain of $\An n$ generated by the shift 
\[
(z_0,z_1,...,z_{kn-1})\mapsto(z_{k},z_{k+1},...,z_{kn-1},z_0,z_1,...,z_{k-1}).
\]
The discrete symplectic action is invariant by this action. Moreover the set of critical points of $\An n$ corresponding to the periodic points in the $\phi$-orbit of $z_0'$ is precisely the $\Z_n$-orbit of $\zz'$. All these critical points share the same Morse index, nullity and local homology. 

In order to conclude the proof, we will show that $\zz''$ does not belong to the  $\Z_n$-orbit of $\zz'$. We proceed by contradiction assuming that this is not true. In particular, $\zz'$ and $\zz''$ have the same local homology, and thus $\Loc_j(\zz'')$ is non-trivial in both degrees $j=dkn-d$ and $j=dkn+d$. This implies $\mor(\zz'')\leq dkn-d$ and $dkn+d\leq \mor(\zz'')+\nul(\zz'')$. Since $\nul(\zz'')\leq 2d$, these two inequalities are actually equalities, i.e. 
\begin{align*}
\mor(\zz'')&=dkn-d,\\
\nul(\zz'')&=2d.
\end{align*}
Therefore, the local homology $\Loc_j(\zz'')$ is non-trivial in maximal degree $j=dkn+d=\mor(\zz'')+\nul(\zz'')$, and according to Proposition~\ref{p:maximal_degree_local_homology}(i) it must be trivial in all the other degrees. This is in contradiction with the fact that $\Loc_{dkn-d}(\zz'')$ is non-trivial.
\end{proof}

\bibliography{_biblio}

\providecommand{\bysame}{\leavevmode\hbox to3em{\hrulefill}\thinspace}
\providecommand{\MR}{\relax\ifhmode\unskip\space\fi MR }
\providecommand{\MRhref}[2]{%
  \href{http://www.ams.org/mathscinet-getitem?mr=#1}{#2}
}
\providecommand{\href}[2]{#2}
\begin{thebibliography}{Abb01}

\bibitem[Abb01]{Abbondandolo:Morse_theory_for_Hamiltonian_systems}
A.~Abbondandolo, \emph{Morse theory for {H}amiltonian systems}, Chapman \&
  Hall/CRC Research Notes in Mathematics, vol. 425, Chapman \& Hall/CRC, Boca
  Raton, FL, 2001.

\bibitem[BK83]{Bangert_Klingenberg:Homology_generated_by_iterated_closed_geodesics}
V.~Bangert and W.~Klingenberg, \emph{Homology generated by iterated closed
  geodesics}, Topology \textbf{22} (1983), no.~4, 379--388.

\bibitem[Bot56]{Bott:On_the_iteration_of_closed_geodesics_and_the_Sturm_intersection_theory}
R.~Bott, \emph{On the iteration of closed geodesics and the {S}turm
  intersection theory}, Comm. Pure Appl. Math. \textbf{9} (1956), 171--206.

\bibitem[Cha84]{Chaperon:Une_idee_du_type_geodsiques_brisees_pour_les_systmes_hamiltoniens}
M.~Chaperon, \emph{Une id{{\'e}}e du type ``g{{\'e}}od{{\'e}}siques
  bris{{\'e}}es'' pour les syst{{\`e}}mes hamiltoniens}, C. R. Acad. Sci. Paris
  S{\'e}r. I Math. \textbf{298} (1984), no.~13, 293--296.

\bibitem[Cha85]{Chaperon:An_elementary_proof_of_the_Conley_Zehnder_theorem_in_symplectic_geometry}
\bysame, \emph{An elementary proof of the {C}onley-{Z}ehnder theorem in
  symplectic geometry}, Dynamical systems and bifurcations ({G}roningen, 1984),
  Lecture Notes in Math., vol. 1125, Springer, Berlin, 1985, pp.~1--8.

\bibitem[Cha93]{Chang:Infinite_dimensional_Morse_theory_and_multiple_solution_problems}
K.-C. Chang, \emph{Infinite-dimensional {M}orse theory and multiple solution
  problems}, Progress in Nonlinear Differential Equations and their
  Applications, 6, Birkh{\"a}user Boston Inc., Boston, MA, 1993.

\bibitem[Con84]{Conley:Lecture_at_the_University_of_Wisconsin}
C.~C. Conley, \emph{{L}ecture at the {U}niversity of {W}isconsin}, April 6,
  1984.

\bibitem[CZ84]{Conley_Zehnder:Subharmonic_solutions_and_Morse_theory}
C.~Conley and E.~Zehnder, \emph{Subharmonic solutions and {M}orse theory},
  Phys. A \textbf{124} (1984), no.~1-3, 649--657.

\bibitem[FH03]{Franks_Handel:Periodic_points_of_Hamiltonian_surface_diffeomorphisms}
J.~Franks and M.~Handel, \emph{Periodic points of {H}amiltonian surface
  diffeomorphisms}, Geom. Topol. \textbf{7} (2003), 713--756 (electronic).

\bibitem[GG09]{Ginzburg_Gurel:Action_and_index_spectra_and_periodic_orbits_in_Hamiltonian_dynamics}
V.~L. Ginzburg and B.~Z. G{\"{u}}rel, \emph{Action and index spectra and
  periodic orbits in {H}amiltonian dynamics}, Geom. Topol. \textbf{13} (2009),
  no.~5, 2745--2805.

\bibitem[GG10]{Ginzburg_Gurel:Local_Floer_homology_and_the_action_gap}
\bysame, \emph{Local {F}loer homology and the action gap}, J. Symplectic Geom.
  \textbf{8} (2010), no.~3, 323--357.

\bibitem[GG12]{Ginzburg_Gurel:Conley_Conjecture_for_Negative_Monotone_Symplectic_Manifolds}
\bysame, \emph{Conley conjecture for negative monotone symplectic manifolds},
  Int. Math. Res. Not. (2012), no.~8, 1748--1767.

\bibitem[GH90]{Guckenheimer_Holmes:Nonlinear_oscillations_dynamical_systems_and_bifurcations_of_vector_fields}
J.~Guckenheimer and P.~Holmes, \emph{Nonlinear oscillations, dynamical systems,
  and bifurcations of vector fields}, Applied Mathematical Sciences, vol.~42,
  Springer-Verlag, New York, 1990, Revised and corrected reprint of the 1983
  original.

\bibitem[Gin10]{Ginzburg:The_Conley_conjecture}
V.~L. Ginzburg, \emph{{T}he {C}onley conjecture}, Ann. Math. \textbf{172}
  (2010), no.~2, 1127--1180.

\bibitem[GM69]{Gromoll_Meyer:On_differentiable_functions_with_isolated_critical_points}
D.~Gromoll and W.~Meyer, \emph{On differentiable functions with isolated
  critical points}, Topology \textbf{8} (1969), 361--369.

\bibitem[Gol01]{Gole:Symplectic_twist_maps}
C.~Gol\'e, \emph{Symplectic twist maps}, Advanced Series in Nonlinear Dynamics,
  vol.~18, World Scientific Publishing Co. Inc., River Edge, NJ, 2001, Global
  variational techniques.

\bibitem[Hei09]{Hein:The_Conley_conjecture_for_irrational_symplectic_manifolds}
D.~Hein, \emph{The {C}onley conjecture for irrational symplectic manifolds}, to
  appear in J. Symplectic Geom., 2009.

\bibitem[Hei11]{Hein:The_Conley_conjecture_for_the_cotangent_bundle}
\bysame, \emph{The {C}onley conjecture for the cotangent bundle}, Archiv der
  Mathematik \textbf{96} (2011), no.~1, 85--100.

\bibitem[Hin09]{Hingston:Subharmonic_solutions_of_Hamiltonian_equations_on_tori}
N.~Hingston, \emph{{S}ubharmonic solutions of {H}amiltonian equations on tori},
  Ann. Math. \textbf{170} (2009), no.~2, 529--560.

\bibitem[HZ94]{Hofer_Zehnder:Symplectic_invariants_and_Hamiltonian_dynamics}
H.~Hofer and E.~Zehnder, \emph{Symplectic invariants and {H}amiltonian
  dynamics}, Birkh{\"a}user Advanced Texts: Basler Lehrb{\"u}cher.
  [Birkh{\"a}user Advanced Texts: Basel Textbooks], Birkh{\"a}user Verlag,
  Basel, 1994.

\bibitem[LC06]{LeCalvez:Periodic_orbits_of_Hamiltonian_homeomorphisms_of_surfaces}
P.~Le~Calvez, \emph{Periodic orbits of {H}amiltonian homeomorphisms of
  surfaces}, Duke Math. J. \textbf{133} (2006), no.~1, 125--184.

\bibitem[LL98]{Liu_Long:An_optimal_increasing_estimate_of_the_iterated_Maslov-type_indices}
C.~Liu and Y.~Long, \emph{An optimal increasing estimate of the iterated
  {M}aslov-type indices}, Chinese Sci. Bull. \textbf{43} (1998), no.~13,
  1063--1066.

\bibitem[LL00]{Liu_Long:Iteration_inequalities_of_the_Maslov-type_index_theory_with_applications}
\bysame, \emph{Iteration inequalities of the {M}aslov-type index theory with
  applications}, J. Differential Equations \textbf{165} (2000), no.~2,
  355--376.

\bibitem[Lon00]{Long:Multiple_periodic_points_of_the_Poincare_map_of_Lagrangian_systems_on_tori}
Y.~Long, \emph{Multiple periodic points of the {P}oincar{\'e} map of
  {L}agrangian systems on tori}, Math. Z. \textbf{233} (2000), no.~3, 443--470.

\bibitem[Lon02]{Long:Index_theory_for_symplectic_paths_with_applications}
\bysame, \emph{Index theory for symplectic paths with applications}, Progress
  in Mathematics, vol. 207, Birkh{\"a}user Verlag, Basel, 2002.

\bibitem[Lu09]{Lu:The_Conley_conjecture_for_Hamiltonian_systems_on_the_cotangent_bundle_and_its_analogue_for_Lagrangian_systems}
G.~Lu, \emph{The {C}onley conjecture for {H}amiltonian systems on the cotangent
  bundle and its analogue for {L}agrangian systems}, J. Funct. Anal.
  \textbf{256} (2009), no.~9, 2967--3034.

\bibitem[Maz11]{Mazzucchelli:The_Lagrangian_Conley_conjecture}
M.~Mazzucchelli, \emph{The {L}agrangian {C}onley conjecture}, Comment. Math.
  Helv. \textbf{86} (2011), no.~1, 189--246.

\bibitem[Mil63]{Milnor:Morse_theory}
J.~Milnor, \emph{Morse theory}, Based on lecture notes by M. Spivak and R.
  Wells. Annals of Mathematics Studies, No. 51, Princeton University Press,
  Princeton, N.J., 1963.

\bibitem[MS98]{McDuff_Salamon:Introduction_to_symplectic_topology}
D.~McDuff and D.~Salamon, \emph{Introduction to symplectic topology}, second
  ed., Oxford Mathematical Monographs, The Clarendon Press, Oxford University
  Press, New York, 1998.

\bibitem[Pal63]{Palais:Morse_theory_on_Hilbert_manifolds}
R.~S. Palais, \emph{{Morse theory on Hilbert manifolds}}, Topology \textbf{2}
  (1963), 299--340.

\bibitem[Poi87]{Poincare:Les_methodes_nouvelles_de_la_mecanique_celeste_Tome_I}
H.~Poincar{{\'e}}, \emph{Les m\'ethodes nouvelles de la m\'ecanique c\'eleste.
  {T}ome {I}. {S}olutions p\'eriodiques. {N}on-existence des int\'egrales
  uniformes. {S}olutions asymptotiques}, Les Grands Classiques
  Gauthier-Villars, Librairie Scientifique et Technique Albert Blanchard,
  Paris, 1987, Reprint of the 1892 original with a foreword by J. Kovalevsky.

\bibitem[RS93a]{Robbin_Salamon:The_Maslov_index_for_paths}
J.~Robbin and D.~Salamon, \emph{The {M}aslov index for paths}, Topology
  \textbf{32} (1993), no.~4, 827--844.

\bibitem[RS93b]{Robbin_Salamon:Phase_functions_and_path_integrals}
\bysame, \emph{Phase functions and path integrals}, Symplectic geometry, London
  Math. Soc. Lecture Note Ser., vol. 192, Cambridge Univ. Press, Cambridge,
  1993, pp.~203--226.

\bibitem[Sal99]{Salamon:Lectures_on_Floer_homology}
D.~Salamon, \emph{Lectures on {F}loer homology}, Symplectic geometry and
  topology ({P}ark {C}ity, {UT}, 1997), IAS/Park City Math. Ser., vol.~7, Amer.
  Math. Soc., Providence, RI, 1999, pp.~143--229.

\bibitem[Sch00]{Schwarz:On_the_action_spectrum_for_closed_symplectically_aspherical_manifolds}
M.~Schwarz, \emph{On the action spectrum for closed symplectically aspherical
  manifolds}, Pacific J. Math. \textbf{193} (2000), no.~2, 419--461.

\bibitem[SZ92]{Salamon_Zehnder:Morse_theory_for_periodic_solutions_of_Hamiltonian_systems_and_the_Maslov_index}
D.~Salamon and E.~Zehnder, \emph{Morse theory for periodic solutions of
  {H}amiltonian systems and the {M}aslov index}, Comm. Pure Appl. Math.
  \textbf{45} (1992), no.~10, 1303--1360.

\end{thebibliography}
\bibliographystyle{amsalpha}

\vspace{0.5cm}

\end{document}